\documentclass[12pt,centertags,oneside]{amsart}
\usepackage{amsmath,amstext,amsthm,amscd,typearea,hyperref}
\usepackage{amssymb}
\usepackage{a4wide}
\usepackage[mathscr]{eucal}
\usepackage{mathrsfs}
\usepackage{typearea}
\usepackage{charter}
\usepackage{pdfsync}
\usepackage[T1]{fontenc}

\usepackage{tikz-cd}
\usepackage{pgfplots}
\tikzset{>=latex}
\pgfplotsset{compat=newest}
\usepackage{caption}
\usepackage{adjustbox}

\usepackage[a4paper,width=16.5cm,top=3cm,bottom=3cm]{geometry}

\numberwithin{equation}{section}

\usepackage[french,english]{babel}

\allowdisplaybreaks
\emergencystretch=\maxdimen
\usepackage{multicol}

\usepackage{xcolor}

\newtheorem{theorem}{Theorem}[section]

\newtheorem{proposition}[theorem]{Proposition}

\newtheorem{lemma}[theorem]{Lemma}
\newtheorem{remark}[theorem]{Remark}

\newtheorem{question}[theorem]{Question}
\newtheorem{conjecture}[theorem]{Conjecture}

\newtheorem*{definition*}{Definition}

\newtheorem{mainthm}{Theorem}

\newcommand{\cali}[1]{\mathscr{#1}}

\newcommand{\dist}{\mathop{\mathrm{dist}}\nolimits}

\newcommand{\dd}{{\rm d}}

\newcommand{\ep}{\epsilon}

\newcommand{\alg}{{\rm alg}}
\newcommand{\dif}{{\rm diff}}

\newcommand{\Area}{{\rm Area}}
\newcommand{\Length}{{\rm Length}}

\newcommand{\Ac}{\cali{A}}
\newcommand{\Bc}{\cali{B}}
\newcommand{\Cc}{\cali{C}}

\newcommand{\T}{\mathbb{T}}
\newcommand{\FS}{{\rm FS}}

\newcommand{\D}{\mathbb{D}}
\newcommand{\C}{\mathbb{C}}

\newcommand{\Z}{\mathbb{Z}}
\newcommand{\R}{\mathbb{R}}

\renewcommand\P{\mathbb{P}}

\newcommand{\lp}{\langle}
\newcommand{\rp}{\rangle}

\newcommand{\norm}[1]{\lVert#1\rVert}

\newcommand{\oA}{\mathcal{A}}
\newcommand{\oB}{\mathcal{B}}

\newcommand{\oN}{\mathcal{N}}

\newcommand{\oC}{\mathcal{C}}

\newcommand{\oT}{\mathcal{T}}

\newcommand{\bA}{\mathbf{A}}
\newcommand{\bB}{\mathbf{B}}

\newcommand{\tx}{\widehat{x}}
\newcommand{\ty}{\widehat{y}}

\newcommand{\tg}{\widehat{g}}

\newcommand{\wth}{\widehat{h}}

\newcommand{\fL}{\mathfrak{L}}
\newcommand{\fA}{\mathfrak{A}}
\newcommand{\fm}{\mathfrak{m}}

\title{Entire curves 
generating all shapes of Nevanlinna currents}

\author{Hao Wu}
\address{Department of Mathematics,  National University of Singapore - 10, Lower Kent Ridge Road - Singapore 119076}
\email{matwu@nus.edu.sg}

\author{Song-Yan Xie}
\address{Academy of Mathematics and System Science \& Hua Loo-Keng Key Laboratory of Mathematics, Chinese Academy of Sciences, Beijing 100190, China;
	School of Mathematical Sciences, University of Chinese Academy of Sciences, Beijing 100049, China}
\email{xiesongyan@amss.ac.cn}

\begin{document}

\date{\today}

\begin{abstract}
First, we show that every complex torus $\mathbb{T}$ contains some entire curve $g: \mathbb{C}\rightarrow \mathbb{T}$  such that the concentric holomorphic discs 
$\{g\,\lvert_{\overline\D_{r}}\}_{r>0}$ 
can generate all the Nevanlinna/Ahlfors currents on $\mathbb T$ at cohomological level. This confirms an anticipation of Sibony.  
Developing further our new method, 
we can  construct  some
twisted entire curve $f: \mathbb{C}\rightarrow \mathbb{CP}^1\times E$ in the product of the rational curve $\mathbb{CP}^1$ and an elliptic curve $E$, such that, concerning Siu's decomposition, 
    demanding any cardinality $|J|\in \mathbb{Z}_{\geqslant 0}\cup \{\infty\}$ and  that $\mathcal{T}_{\dif}$ is trivial ($|J|\geqslant 1$)  or not ($|J|\geqslant 0$), we can always find a sequence of concentric holomorphic discs $\{f\,\lvert_{\overline\D_{r_j}}\}_{j
    \geqslant 1}$ to generate a Nevanlinna/Ahlfors current  $\mathcal{T}=\mathcal{T}_{\alg}+\mathcal{T}_{\dif}$ with the singular part $\mathcal{T}_{\alg}=\sum_{j\in J} \,\lambda_j\cdot[\mathsf C_j]$ in the desired shape.
    This  fulfills the missing case where $|J|=0$ in the previous work of Huynh--Xie. 
    By a result of Duval, each $\mathsf C_j$ must be rational or elliptic. We will show that there is no \textsl{a priori} restriction  on the numbers of rational and elliptic  components in the support of $\mathcal{T}_{\alg}$, thus answering a question of Yau and  Zhou. Moreover, we will show that  the positive coefficients $\{\lambda_j\}_{j\in J}$ can be arbitrary as long as the total mass of $\oT_{\alg}$ is less than or equal to $1$. 
    Our results foreshadow  striking holomorphic flexibility of entire curves in Oka geometry, which deserves further exploration.
\end{abstract}

\clearpage\maketitle
\thispagestyle{empty}


\smallskip



\section{\bf Introduction}\label{intro-sec}

One central problem in complex geometry is the famous Green-Griffiths conjecture~\cite{MR0609557}, which
 stipulates that for every complex projective
variety $X$ of general type, all entire curves $f: \mathbb{C}\rightarrow X$ shall be factored through certain  proper algebraic subvariety $Y\subsetneq X$. 
Some interest also comes from the philosophical analogy  
 between the value distribution of entire curves in Nevanlinna theory and the  locations of rational points in  Diophantine geometry
(cf.~\cite{MR0828820, Vojta, Xie-Yuan}), best illustrated as ``geometry governs arithmetic''.

In his celebrated solution~\cite{Mcquillan1998} to the Green--Griffiths conjecture for complex projective surfaces satisfying a Chern number inequality $c_1^2>c_2$, McQuillan introduced the so called {\em Nevanlinna currents} to capture asymptotic behaviors of entire curves. 
Later, simplified analogues  called {\em Ahlfors currents} were employed by Duval~\cite{Duval2008} to obtain a deep, quantitative Brody Lemma~\cite{MR0470252}. One important corollary is a characterization of complex hyperbolicity in terms of linear isoperimetric
inequality for holomorphic discs (see also~\cite{Kleiner-preprint}), which was anticipated by Gromov~\cite[p.~152 (c)]{MR1826251}. 
In value distribution theory and complex dynamics, 
Nevanlinna/Ahlfors currents also found important applications,  
see for instance~\cite{Dinh-Sibony2018,  Duval-Huynh2018, Huynh-Vu2020}. 

\smallskip

Now, let us recall the precise definitions.  Given  a compact complex manifold  $X$ with a Hermitian form $\omega$. 
Let $\{F_j: \overline\D (z_j, R_j)\to X\}_{j\geqslant 1}$ be a sequence of nonconstant holomorphic discs,  smooth up to the boundary,
where each $\overline{\mathbb{D}}(z_j, R_j)\subset \mathbb{C}$ is the closed disc centered at $z_j$ with radius $R_j$. 
We can associate  each $F_j$  with a  positive current $\widehat\oN_{F_j}$ of bidimension $(1,1)$,  
which evaluates every smooth $(1,1)$-form $\phi$ on $X$ by a Jensen-type formula
\[
\lp \widehat\oN_{F_j}, \phi\rp\,
=
\,
\int_{0}^{R_j}\,\frac{\dd t}{t}\int_{\mathbb{D}(z_j, t)}\,
F_j^*\phi.
\]
Write $\fA_{F_j}(z_j,R_j):= \lp \widehat\oN_{F_j}, \omega\rp$.
 Consider the sequence of normalized currents
 \begin{equation*}
\label{Nevanlinna currents}
\Big\{
	\oN_{F_j}:=\frac{1}{\fA_{F_j}(z_j,R_j)}\cdot \widehat\oN_{F_j}
\Big\}_{j\geqslant 1}
\end{equation*} 
of  mass $1$. 
By Banach–Alaoglu’s theorem, certain subsequence 
converges in weak topology to some positive current
$\oN$. 
If  a priori $\{F_j\}_{j\geqslant 1}$ satisfies
the {\em length-area condition}
\begin{equation}
	\label{length-area condition}
	\lim_{j\rightarrow \infty}	
	\dfrac{\fL_{F_j}(z_j,R_j)}{\fA_{F_j}(z_j,R_j)}
	=
	0,  \quad \text{where}\quad \fL_{F_j}(z_j,R_j)
:=
\int_0^{R_j}\,
\Length_{\omega}\,\big(F_j(\partial \mathbb{D}(z_j, t))\big)\,
\frac{\dd t}{t},
\end{equation}
then one can check that $\oN$ is in fact closed (cf.~\cite{MR1989205}), and it is called a Nevanlinna current.  When $z=0$, we abbreviate $\fL_{F}(z,r),\fA_{F}(z,r)$  as $\fL_{F}(r),\fA_{F}(r)$.

In particular, given an entire curve $f: \mathbb{C}\rightarrow X$,  Ahlfors Lemma~(cf.~\cite[p.~55]{Huynh2016}) ensures that we 
can always select some increasing radii $\{r_j\}_{j\geqslant 1}$ tending to infinity  such that $\{f\,\lvert_{\overline \D _{r_j}}\}_{j\geqslant 1}$
enjoys~\eqref{length-area condition}, where $\overline{\mathbb{D}}_{r_j}:=\{z\in \mathbb{C}: |z|\leqslant r_j\}$. Thus we receive some Nevanlinna current associated with  $f$. 

\smallskip

The definition of Ahlfors currents
is likewise and simpler. Again we start with
 a sequence
 $\{F_j: D_j\to X\}_{j\geqslant 1}$ 
 of nonconstant holomorphic discs, smooth up to the
boundary, satisfying another  length-area condition that
\begin{equation}
	\label{length-area-condition-Ahlfors}
\lim_{j\rightarrow \infty}\,	\dfrac{\Length_{\omega}(F_j(\partial D_j))}{\Area _{\omega}(F_j(D_j))}
=
 0.
\end{equation}
From the sequence of normalized currents
\begin{equation*}
\label{Ahlfors currents}
\Big\{
	\oA_j := \frac{1}{ \Area_{\omega} (F_j(D_j))} \cdot  (F_j)_*  [D_j] 
\Big\}_{j\geqslant 1}
\end{equation*} 
of mass $1$,
by compactness and diagonal argument, 
after passing to some subsequence, in the limit 
we receive some positive current of bidimension $(1, 1)$,   which is actually closed by~\eqref{length-area-condition-Ahlfors}. 

Specifically, given an entire curve $f: \mathbb{C}\rightarrow X$, by Ahlfors Lemma (cf.~\cite[p.~7]{duval-pan}), one
can find some increasing radii $\{r_j\}_{j\geqslant 1}$ tending to infinity such that $\{f\,\lvert_{\overline \D_{r_j}}\}_{j\geqslant 1}$
satisfies the length-area condition~\eqref{length-area-condition-Ahlfors}.
Thus $f$ produces some Ahlfors current. 

\smallskip

One advantage of
Nevanlinna currents 
over Ahlfors currents is that,
when $f: \mathbb{C}\rightarrow X$ is algebraically nondegenerate, 
its associated Nevanlinna currents
 are always nef   (cf.~\cite{MR1989205}).

Nevanlinna/Ahlfors curents associated with an entire curve $f: \mathbb{C}\rightarrow X$ can be regarded as certain noncompact Poincar\'e dualities of $f(\mathbb{C})$. 
Hence several classical results in value distribution theory can be reformulated in terms of intersections of corresponding cohomology classes. For instance, the First Main Theorem of Nevanlinna theory can be expressed as an inequality between the algebraic intersection and the geometric intersection (cf.~\cite{Duval-Huynh2018}).
Therefore, 
it would be natural and fundamental to ask

\begin{question}\label{uniqueness question}\rm
    Are all Nevanlinna/Ahlfors currents associated with the same entire curve cohomologically equivalent?
\end{question}

Meanwhile, in a simplied proof~\cite{brunella1999} of McQuillan's result~\cite{Mcquillan1998}, a key trick of Brunella is to use Siu's decomposition theorem~\cite{Siu1974} to  write every positive closed Nevanlinna current
$\mathcal{N}$ as $\mathcal{N}_{\alg }+\mathcal{N}_{\dif}$,  
where the singular part $\mathcal{N}_{\alg }=\sum_{j\in J} \,\lambda_j\cdot[\mathsf{C}_j]$
is some positive linear combination ($\lambda_j>0$; $J\subset \mathbb{Z}_+$, could be empty) of currents of integration on distinct irreducible
algebraic curves $\mathsf{C}_j$, and where the diffuse part $\mathcal{N}_{\dif}$ is a positive closed $(1, 1)$-current having zero Lelong 
number along any algebraic curve.
The idea was to handle the two parts
$\mathcal{N}_{\alg },\mathcal{N}_{\dif}$ separately by distinct techniques from algebraic geometry and complex dynamics respectively. To see the necessity, he raised

\begin{conjecture}[\cite{brunella1999}--p.~200] 
\label{Brunella conjecture 1}\rm
 Some entire curve shall produce a Nevanlinna current $\mathcal{N}=\mathcal{N}_{\alg }+\mathcal{N}_{\dif}$ with both nontrivial singular part $\mathcal{N}_{\alg }\neq 0$ and nontrivial diffuse part $   \mathcal{N}_{\dif}\neq 0$.
\end{conjecture}

Question~\ref{uniqueness question} (answer: no)
and Conjecture~\ref{Brunella conjecture 1} (answer: yes)
were solved  by Huynh and the second named author~\cite{huynh-xie-JMPA} by constructing explicit exotic examples.  

\begin{theorem}[\cite{huynh-xie-JMPA}]
\label{theorem HX21}
    There exists an entire curve $f: 
    \mathbb{C}\rightarrow X$ such that, 
    given any cardinality $|J|\in \mathbb{Z}_+\cup \{\infty\}$ and any a priori requirement that $\mathcal{T}_{\dif}$ is trivial or not, by taking certain increasing radii $\{r_j\}_{j\geqslant 1}$ tending to infinity, the sequence of holomorphic discs $\{f\,\lvert_{\overline\D_{r_j}}\}_{j
    \geqslant 1}$ yields a Nevanlinna/Ahlfors current  $\mathcal{T}=\mathcal{T}_{\alg }+\mathcal{T}_{\dif}$ with the desired shape $\mathcal{T}_{\alg }=\sum_{j\in J} \,\lambda_j\cdot[\mathsf{C}_j]$ in Siu's decomposition.
    {}
\end{theorem}

Here $X$ is some $\mathbb{CP}^1$-bundle over an elliptic curve. The construction of $f$ 
is \textsl{ad hoc}. For instance, 
it relies crucially on the rotational symmetry $\sqrt{-1}\cdot \Gamma=\Gamma$ of the  Gaussian integer lattice  $\Gamma=\mathbb{Z}[\sqrt{-1}]$~(\cite[Lemma 3.2]{huynh-xie-JMPA}). Moreover, it
has the feature that all the
obtained Nevanlinna/Ahlfors currents charge positive mass on certain unique elliptic curve  
$\mathsf{C}_{\infty}$ in $X$,
hence failing to embrace the purely diffuse case where $|J|=0$. In hindsight, the existence of such $f$ reflects the fact that $X$ is an Oka manifold~\cite{Forstneric-Book}.

\medskip
After receiving an early manuscript of
\cite{huynh-xie-JMPA}, 
Sibony asked

\begin{question}[\cite{Sibony-email}]
\label{Sibony-question}\rm
	Show that for $Y=\mathbb{CP}^n$, or complex torus, all Ahlfors
	 currents can be obtained by single entire curve $f: \mathbb{C}\rightarrow Y$.	
\end{question}

Using holomorphic discs 
$\{f\,\lvert_{ \overline \D(z, r)}\}_{z\in \mathbb{C}, r>0}$ 
not necessarily centered at the origin, Sibony's Question~\ref{Sibony-question} has been confirmed by the second named author.  In fact, such phenomenon holds 
for a large class of compact complex manifolds with certain weak Oka property~\cite[Theorem A]{Xie-2023}.  For testing the Oka principle~\cite{Forstneric-Book}, 
we may also extend Sibony's Question~\ref{Sibony-question} as follows.

\begin{conjecture}[\cite{Xie-2023}]
\label{Xie's conjecture}
\rm
Let $X$ be a compact Oka manifold.
	There should be some entire curve $g: \mathbb{C}\rightarrow X$ such that  $\{g\,\lvert_{\overline \D_r}\}_{r>0}$  generates all Nevanlinna/Ahlfors currents on $X$.
\end{conjecture}

 In this paper, we develop further a technique introduced in~\cite[Sect.\ 3]{Xie-2023}  to verify Conjecture~\ref{Xie's conjecture} for any complex torus, hence answering Sibony's Question~\ref{Sibony-question} using only concentric holomorphic discs. 

\begin{mainthm}\label{thm-class}
There exists an entire curve $\mathbf F: \mathbb{C}\rightarrow \mathbb{T}$ into any given complex torus $\mathbb{T}$, such that the concentric holomorphic discs
$\{\mathbf F\,\lvert_{ \overline \D_R}\}_{R>0}$ 
can generate all the Nevanlinna/Ahlfors currents on $\T$, up to cohomological equivalence.  
\end{mainthm}

It foreshadows 
striking  holomorphic flexibility of entire curves in Oka geometry. 
The construction of such entire curves $\mathbf F$ is based on a sophisticated induction process. The key ingredients are Hodge theory on complex tori,
the spectral theorem for Hermitian  matrices, and delicate area-growth estimates.

\medskip
 Recall a result~\cite{Duval2006} of Duval that   each 
irreducible component $\mathsf{C}_j$
in the singular part
$\mathcal{T}_{\alg }=\sum_{j\in J} \,\lambda_j\cdot[\mathsf{C}_j]$ of an Nevanlinna/Ahlfors current $\mathcal{T}=\mathcal{T}_{\alg }+\mathcal{T}_{\dif}$ must be  rational or elliptic. Thus we can rewrite
$\mathcal{T}_{\alg}$ as
 \begin{equation}
     \label{aimed singular part}
     \sum_{j\in J_{\mathsf{E}}} 
     \,c^{\mathsf{E}}_{j}\cdot[\mathsf{E}_j]+\,
    \sum_{k\in J_{\mathsf{R}}} \,c^{\mathsf{R}}_k\cdot[\mathsf{R}_k]
 \end{equation}
     for some distinct irreducible elliptic curves  $\mathsf{E}_j$ and some distinct
     rational curves $\mathsf{R}_k$, and for some strictly positive constants $c_j^{\mathsf E}, c_k^{\mathsf R}$   ($J_{\mathsf{E}}, J_{\mathsf{R}}$ are countable sets, could be empty). 
     Thus we can 
classify Nevanlinna/Ahlfors currents  in different shapes by the data
\begin{equation}
    \label{all shapes data}
\big(|J_{\mathsf{E}}|\in \mathbb{Z}_{\geqslant 0}\cup \{\infty\}, \,
|J_{\mathsf{R}}|\in \mathbb{Z}_{\geqslant 0}\cup \{\infty\}, \,
\mathcal{T}_{\dif} \text{ is trivial/nontrivial}\big).
\end{equation}

    \smallskip 
The next question is very natural, and was raised by 
  Shing-Tung Yau and
Xiangyu Zhou independently
to the second named author around the same time (January 2021).

\begin{question}[Yau, Zhou]
	\label{Yau-question-2}\rm
 Is there any  a priori restriction on  the numbers $|J_{\mathsf{E}}|$, $|J_{\mathsf{R}}|$ of elliptic and  rational components in the singular parts of Nevanlinna/Ahlfors currents? 
 \end{question}

Let $E=\C/\Gamma$ be 
 an elliptic curve,  where $\Gamma\subset \C$ is some lattice.  
 Set $X:=\mathbb{CP}^1\times E$. Denote by
 $\pi_1: X\rightarrow \mathbb{CP}^1$, $\pi_2: X\rightarrow E$
 the projections onto the first and the second factors respectively.
 Let 
 $\omega_\FS$ be the Fubini-Study form on $\C\P^1$ giving unit area
 $\int_{\C\P^1}\omega_\FS=1$.  Let $\omega_E$ be the Hermitian form on $E$ induced  from the standard Euclidean form $\omega_{\C}:={i\over 2} \,\dd z\wedge \dd \overline z$ on $\C$, with area
 $\int_E \omega_E=:\varrho$. 
 Here  is a strengthening of Theorem~\ref{theorem HX21}.

\begin{mainthm}\label{thm-shape}	There exists an entire curve $\mathbf{F}:\C\rightarrow X$ producing all shapes~\eqref{all shapes data} of Nevanlinna/Ahlfors currents
by $\{\mathbf F\,\lvert_{ \overline \D_R}\}_{R>0}$. 
\end{mainthm}

 In particular, this answers Yau and Zhou's Question~\ref{Yau-question-2}
completely.
In fact, we can show that, for any singular current $T$ of the shape~\eqref{aimed singular part}, where
$\mathsf{E}_j=\pi_1^{-1}(x'_j)$, $\mathsf{R}_k=\pi_2^{-1}(y'_k)$ are fibers of  
$\pi_1$, $\pi_2$ for arbitrary distinct points $\{x_j'\}_{j=1}^{|J_{\mathsf{E}}|}\subset\C\P^1$, $\{y_k'\}_{k=1}^{|J_{\mathsf{R}}|}\subset E$, 
and where the coefficients $c^{\mathsf{E}}_{j}, c^{\mathsf{R}}_k>0$ satisfy that the total mass $\varrho\cdot\sum_{j\in J_{\mathsf{E}}} c^{\mathsf{E}}_{j}+
    \sum_{k\in J_{\mathsf{R}}} c^{\mathsf{R}}_k$ of $T$ is less than or equal to $1$, then there exists some sequence 
    of increasing radii $\{R_s\}_{s\geqslant 1}$ tending to infinity such that
    $\{\mathbf F\,\lvert_{ \overline \D_{R_s}}\}_{s\geqslant 1}$
 generates a Nevanlinna/Ahlfors current $\oT=\oT_{\alg}+\oT_{\dif}$ having the aimed singular part $\oT_{\alg}=T$. This is a  surprising result!
 Previously we would  hesitate to imagine such tremendous holomorphic flexibility of entire curves. 

 The proof of Theorem~\ref{thm-shape} is quite challenging. First of all, we need to design an intriguing model of entire curves with parameters.
Next, we have to determine the  parameters properly, see Proposition~\ref{Area distribution}. One natural idea is to make use of the Brouwer fixed point theorem. As a matter of fact, that was our original approach,
which was indeed successful though involved. 
Nevertheless,  we managed to
obtain quite explicit formulas about these parameters by using a trick --- see~\eqref{Duval's trick},~\eqref{Duval's trick 2} --- in spirit of Ahlfors' covering surface theory (cf.~\cite{MR0279280, MR3342651}).  
In general, our strategy  could work  as well for any compact Oka manifold $Y$ containing infinitely many
rational and elliptic curves respectively, by constructing a ``good''  auxiliary holomorphic map $\C^n\rightarrow Y$ for some $n\geqslant 1$.  We believe that such an approach
will be  a key towards Conjecture~\ref{Xie's conjecture}.

\medskip
\noindent\textbf{Notation:} Throughout this paper, the symbols $\lesssim$ and $\gtrsim$ stand for inequalities up to a positive multiplicative constant.  The dependence of these constants on certain parameters, or lack thereof, will be clear from the context.  The symbols $\gg$ and $\ll$ mean much larger and much smaller respectively.

\bigskip\noindent
{\bf Acknowledgments.}
This work is inspired by insightful questions of Sibony, Yau, and Zhou.
We thank Dinh Tuan Huynh, Bin Guo (AMSS), Lei Hou,  Yi C. Huang for helpful suggestions which improved the exposition.

\bigskip\noindent
{\bf Funding.} H. Wu is supported by 
the NUS and MOE grants A-0004285-00-00 and MOE-T2EP20120-0010.
S.-Y. Xie is
partially supported by 
National Key R\&D Program of China Grant
No.~2021YFA1003100 and  NSFC Grant No.~12288201.

\medskip

\section{\bf Twisted entire curves in complex tori producing all Nevanlinna currents}\label{sect. 2}

In this section, we will construct an entire curve $\mathbf F:\C\to \T$ for Theorem~\ref{thm-class}. The key ingredient is Hodge theory (cf.~\cite{Griffiths-Harris, MR1967689}) on complex tori, as we present now.

\smallskip
Let $X$ be a compact  K\"ahler manifold of complex dimension $n$. 
The \textit{de Rham cohomology group} $H^l(X,\C)$ for every $0\leqslant l\leqslant 2n$ is  the quotient of the $\mathbb{C}$-linear space of  closed $l$-forms (resp. $l$-currents) by the subspace of exact $l$-forms (resp. $l$-currents). We can also define the real group $H^l(X,\R)$ using  real forms or currents. The \textit{Hodge cohomology group} $H^{p,q}(X,\C)$ for each $0\leqslant p,q\leqslant n$ is the subspace of $H^{p+q}(X,\C)$ generated by the class of closed $(p,q)$-forms  (resp. currents). Hodge decomposition states that 
$$H^l(X,\C)=\oplus_{p+q=l}\, H^{p,q}(X,\C).$$
The Poincar\'e duality gives a perfect pairing 
\begin{align*}
H^{p,q}(X,\C) 
\times H^{n-p,n-q}(X,\C)
\longrightarrow 
\mathbb{C},
\quad
([\phi_1], [\phi_2])
\mapsto
\int_{X}
\phi_1\wedge \phi_2.
\end{align*}
When $p=q$, define 
$$ H^{p,p}(X,\R):=H^{p,p}(X,\C)\cap H^{2p}(X,\R).$$
It is clear that all Nevanlinna/Ahlfors on $X$ currents are in  $H^{n-1,n-1}(X,\R)$. 

\smallskip

Let $\omega$ be the Hermitian form on $\mathbb{T}$ induced by the standard Euclidean form $\omega_{\C^n}:=\frac{i}{2} \,\sum_{j=1}^n \dd z_j\wedge \dd \overline z_j$  through the canonical projection $\Pi: \C^n\rightarrow \mathbb{T}=\C^n/\Lambda$.
 The complex vector space $H^{n-1,n-1} (\T,\C)$ admits a $\mathbb{C}$-linear basis (cf.~\cite[p.~301]{Griffiths-Harris})
\[
\big\{\dd z_I\wedge \dd \overline{z}_J:\,
I, J\subset \{1, \dots, n\}, |I|=|J|=n-1 \big\}
\]
 induced from $\mathbb{C}$-coefficient $(n-1, n-1)$-forms on $\C^n$,
where for $I=\{i_1, \dots, i_{n-1}\}$ and $J=\{j_1, \dots, j_{n-1}\}$, 
we abbreviate
$\dd z_I:=
\dd z_{i_1}\wedge\cdots \wedge \dd z_{i_{n-1}}$ and
$\dd \overline{z}_J:=
\dd \overline{z}_{j_1}\wedge\cdots \wedge \dd \overline{z}_{j_{n-1}}$.

\smallskip

For any $\vec{v}=(v_1,v_2,\dots,v_n)\in \mathbb{C}^n\setminus \{0\}, \vec{c}\in \mathbb{C}^n$, we will show that
the affine entire curve $F_{v,c}(z)=\Pi(z\cdot \vec{v}+\vec{c})$ from $\mathbb{C}$ to 
 $\mathbb{T}$ produces a unique Nevanlinna/Ahlfors current, denoted by
 $\oT_{[\vec{v}]}$,  up to cohomological equivalence in $H^{n-1,n-1}(\T,\R)$. 
  To see that $\oT_{[\vec{v}]}$  is unique, independent of $\vec{c}$, and only depends on the equivalent class $[\vec{v}]\in \mathbb{CP}^{n-1}$, 
 we need the following

\smallskip\noindent
{\bf Observation.} 
For any $(1,1)$-class $\phi$ in $H^{1,1}(\T,\C)$ induced by the  $(1,1)$-form $\sum_{k,\ell=1}^n\xi_{k,\ell} \, \dd z_k \wedge \dd \overline z_\ell$ on $\C^n$, where $\xi_{k,\ell}\in \mathbb{C}$,
one has a neat formula
\begin{equation}\label{value-T-v}
\lp \oT_{[\vec{v}]}, \phi\rp
=-2i \sum_{k,\ell=1}^n   v_k \, \xi_{k,\ell} \, \overline v_\ell /\norm{\vec{v}}_{\mathbb{C}^n}^2,
\end{equation}
where $\lVert\cdot\rVert_{\mathbb{C}^n}$ stands for the stancard Euclidean norm on $\mathbb{C}^n$.  
Hence the uniqueness of $\oT_{[\vec{v}]}$ in $H^{n-1,n-1}(\T,\C)$ follows from the Poincar\'e duality.

\begin{proof}
 For any disc $\D_R$ with radius $R>0$, we compute directly
\[
    \lp (F_{v,c})_*[ \D_R ] , \phi \rp  
= \int_{\D_R} \,F_{v,c}^*\,\phi
=\sum_{k,\ell=1}^n \int_{\D_R}  \xi_{k,\ell} \,   \dd (v_k z) \wedge \dd\overline {(v_\ell z)}=-2i \pi R^2 \sum_{k,\ell=1}^n   v_k \,\xi_{k, \ell} \, \overline  v_\ell,
\]
\[
\Area_\omega\big( F_{v,c}( \D_R ) \big) =\int_{\D_R}   F_{v,c}^*\,\omega = \sum_{j=1}^n\int_{\D_R}  \frac{i}{2} \, \dd (v_jz) \wedge \dd\overline {(v_j z)} =\pi R^2\,\norm{\vec{v}}_{\C ^n}^2.
\]
Hence~\eqref{value-T-v} follows for Ahlfors currents. Similarly, 
we can show~\eqref{value-T-v}  for Nevanlinna currents.
\end{proof}

In the linear space of bidimension $(1, 1)$-currents on $\mathbb{T}$, let $\oC$ be the convex cone generated by $\{\oT_{[\vec{v}]}\}_{[\vec{v}]\in \mathbb{CP}^{n-1}}$.
Let $\oC_{1}\subset \oC$ be the subset consisting of mass $1$ currents with respect to $\omega$, i.e., $\oC_{1}:=\{T\in \oC : \lp T,\omega\rp=1\}$.
One can check that 
$\oT_{[\vec{v}]}\in \oC_{1}$ for every  
$[\vec{v}]\in \mathbb{CP}^{n-1}$.

\begin{proposition}
\label{key observation}
    Any Nevanlinna/Ahlfors current $\oT$ on $\mathbb{T}$ can be represented by an element
$T'$ in $\oC_1$ in the same cohomology class $[T']=[\oT]$. Moreover,  $T'$ can be chosen in the shape
\[
T'\,
=\,
\sum_{s=1}^n\,\beta_s\cdot \oT_{[\vec{v}_s]}
\]
for some
orthonormal basis $\{\vec{v}_{s}\}_{s=1}^n$ 
of $\mathbb{C}^n$
and  some non-negative coefficients $\beta_1, \dots, \beta_n$ with $\sum_{s=1}^n \beta_s=1$. Both $\{\vec{v}_{s}\}_{s=1}^n$ and $\{\beta_{s}\}_{s=1}^n$ depend on $\oT$.
\end{proposition}

\begin{proof}
For any two vectors $\vec u:=(u_1,u_2\dots,u_n), \vec w:=(w_1,w_2,\dots,w_n)$ in $\C^n$, denote by $\phi_{\vec u, \vec w}$ the $(1,1)$-form on $\T$ induced by
$\frac{i}{2}  \,\sum_{k=1}^n u_k \dd z_k\wedge   \sum_{\ell=1}^n  \overline w_\ell \dd \overline z_\ell$ from $\C^n$.
 Consider the  
 bilinear form 
    $$\oB(\vec u,\vec w)\,
    :=\,
    \lp 
    \oT, 
    \phi_{\vec u, \vec w} \rp,$$
    which is clearly  semi-positive and Hermitian. Hence 
by the spectral theorem (cf. e.g.~\cite[p.~582, Theorem~6.4]{MR1878556}),
we can find some orthonormal basis
$\{\vec{v}_1, \dots, \vec{v}_n\}$
of $\mathbb{C}^{n}$, 
such that $\oB$ reads as
\[
\oB(\vec u,\vec w)\,
=\,
\sum_{s=1}^n\,
\beta_s\,
\lp \vec u, \vec{v}_s\rp 
\cdot
\overline{
\lp \vec w, \vec{v}_s\rp 
}
\qquad
(\forall\, \vec u, \,\vec w \,\in \,\mathbb{C}^n)
\]
for some non-negative numbers $\beta_1, \dots, \beta_n\geqslant 0$. 
Moreover, \eqref{value-T-v} verifies that
\[
\sum_{s=1}^n\,
\beta_s\,
\lp \vec u, \vec{v}_s\rp 
\cdot
\overline{
\lp \vec w, \vec{v}_s\rp 
}
\,=\,
\lp T', 
\phi_{\vec u, \vec w}
\rp
,
\]
where $T':=\sum_{s=1}^n\,\beta_s\cdot \oT_{[\vec{v}_s]}$.
Thus $\oT$ is cohomologous to $T'$ by Hodge theory.

Lastly,
noting that
that $\omega$ can be rewritten as $\sum_{s=1}^n \phi_{\vec v_s, \vec v_s}$,  and that
$\oT(\omega)=1$ by the definition of Nevanlinna/Ahlfors currents, using \eqref{value-T-v} again,
we receive
$$
1
=
\lp \oT, \omega\rp
=
\sum_{s=1}^n \lp \oT, \phi_{\vec v_s, \vec v_s}\rp =
\sum_{s=1}^n \beta_s. $$ 
Hence we conclude the proof.
\end{proof}

\smallskip
\noindent
{\bf Runge's approximation theorem.} 
{\it
Let $K$ be a compact  subset of $\C$ such that $\C\setminus K$ is connected. Then any  holomorphic function $f$ defined on a neighborhood of $K$ can be approximated uniformly on $K$ by a sequence of polynomials.
}
\smallskip

The following two consequences of Runge's approximation theorem will be used later.

\begin{lemma}\label{lem-exist}
	For every $x\in \partial \D_R$ and a small open set $U$ containing $x$, there exists an entire function $\varphi$ depending on $R, x, U$, such that $|\varphi(x)|>1$ and $|\varphi|<1$ on $\overline \D_R\setminus U$. 
\end{lemma}

\begin{proof}
By~\cite[Key Lemma]{Xie-2023}, we can find some  holomorphic function $\widehat \varphi$ defined on a neighborhood of $\overline \D_R$ such that $|\widehat\varphi(x)|>1$ and $|\widehat\varphi|<1$ on $\overline \D_R\setminus U$. Now, Runge's approximation theorem  concludes the proof.
\end{proof}

\begin{remark}\rm
Here is an explicit elegant construction communicated to us by Fusheng Deng. For some sufficiently small  $\epsilon>0$, the entire function
$
\varphi(z):=
\exp(z/x-1+\epsilon)
$
satisfies the desired properties.
\end{remark}

\begin{lemma}\label{lem-disc-wp-pi}
Let $\{D_j\}_{1\leqslant j\leqslant s}$ be a family of disjoint close discs in $\C\setminus \overline\D_R$ for some $R>0$. Then for any entire function $f$, for any prescribed values $z_1, \dots, z_s\in \C$, 
there exists some entire function $g$, such that 
$|g-f|<\delta$ on $\overline\D_R$, and that 
$|g-z_j|<\delta$ 
on $D_j$ for every $j=1, \dots, s$.
\end{lemma}

\begin{proof} 
Apply Runge's approximation theorem to $K=\overline \D_R \bigcup\big(\cup_{j=1}^s D_j\big)$.
\end{proof}

\smallskip

From now on, we fix
\begin{equation}
    \label{fix an enumeration of C_1,n}
\text{a countable  subset 
$\{\oT_{\mathfrak m}\}_{\mathfrak m \geqslant  1}$ of\, $\mathcal{C}_{1}$},
\end{equation}
which generates, 
in the finite dimensional cohomology group $H^{n-1,n-1}(\T,\R)$, dense image with respect to that  of
$\mathcal{C}_{1}$.  We also fix a bijection $$(\sigma, \tau): 
\mathbb{Z}_+
\longrightarrow
\mathbb{Z}_+ \times \mathbb{Z}_+,$$ 
so that $\sigma: \mathbb{Z}_+\rightarrow  \mathbb{Z}_+$ is surjective and every fiber $\sigma^{-1}(\mathfrak m)$ contains infinitely many integers for $\mathfrak m\geqslant 1$. 
Our desired entire curve 
$\mathbf F=\Pi\circ \widehat {\mathbf F} : \mathbb{C}\rightarrow \mathbb{T}$ 
will be obtained by some lifting
$\widehat {\mathbf F}: \mathbb{C}\rightarrow \mathbb{C}^n$, which is the limit of a sequence of entire curves
$\widehat F_{t}: \mathbb{C}\rightarrow \mathbb{C}^n$ designed by the following sophisticated algorithm.

\medskip\noindent
{\bf Step $0$.} Set  $R_0:=2n$. Let  $\widehat F_0: \mathbb{C}\rightarrow \mathbb{C}^n$ be any entire curve. Choose  $\epsilon_0:=1$.

\medskip

\noindent In {\bf Step $t-1$} for $t\geqslant 1$, we record certain entire curve $\widehat F_{t-1}$ on $\C^n$, a large radius $R_{t-1}$ and a small error bound $\ep_{t-1}$.

\medskip\noindent
{\bf Step $t$}  for $t
\geqslant 1${\bf .}
Read $\mathfrak{m}=\sigma(t)$. 
By Proposition~\ref{key observation},
we can rewrite 
$\oT_{\mathfrak{m}}$ in~\eqref{fix an enumeration of C_1,n} as $\sum_{j=1}^n\,\beta_{t, j}\cdot \oT_{[\vec{v}_{t, j}]}$ for some non-negative constants $\beta_{t, 1}, \dots, \beta_{t, n}$ and
some orthonormal basis $\{\vec{v}_{t, s}\}_{s=1}^n$ 
of $\mathbb{C}^n$. 
Set $R_{t}:=2R_{t-1}$.  

\smallskip

We can sparsely select
$n$ distinct points $x_{t, 1}, \dots, x_{t, n}$ on $\partial \mathbb{D}_{R_t}$ such that the closed unit discs $ \{\overline{\mathbb{D}}(x_{t, j}, 1)\}_{j=1}^n$ are pairwise disjoint. 
By Lemma \ref{lem-exist}, for the point $x_{t,1}$ and its neighborhood $\D(x_{t,1}, 1)$, we can find some entire function $\varphi_{t, 1}$ such that 
\begin{equation}\label{condition-varphi-}
|\varphi_{t, 1}(x_{t, 1})|>1  \quad\text{and}\quad |\varphi_{t, 1}|<1 \quad \text{on}\quad \overline \D_{R_t}\setminus \D(x_{t, 1}, 1).
\end{equation}
Since  $x_{t, 1}, \dots, x_{t, n}$ are in the same circle around the origin, 
by composing $\varphi_{k, 1}$ with some rotations,
 we can also define $\varphi_{t, j}(z):=  \varphi_{t, 1} (e^{i\theta_{t, j}}\cdot z)$, for $j=2, \dots, n$, with likewise property
\[
|\varphi_{t, j}(x_{t, j})|>1  \quad\text{and}\quad |\varphi_{t, j}|<1 \quad \text{on}\quad \overline \D_{R_t}\setminus \D(x_{t, j}, 1).
\]
Here, each angle $\theta_{t, j}\in [0, 2\pi)$ is uniquely determined by
$e^{i\theta_{t, j}}\cdot x_{t, j}=x_{t, 1}$. 

\smallskip

Next, we define an entire curve
\begin{equation}\label{define F_k}
\widehat F_{t}\,:=\,
\widehat F_{t-1} 
+  
\sum_{j=1}^{n}\,{\sqrt{\beta_{t, j}} \over \alpha_t}\cdot\varphi_{t, j}  ^{\alpha_t}\cdot
\vec{v}_{t, j}
\end{equation}
 in $\C^n$, 
where  the integer exponent  $\alpha_t\gg 1$ is chosen to meet the following sophisticated requirements.

\begin{itemize}
    \item [\bf ($\spadesuit$).]
    The entire curve $F_t:=\Pi\circ \widehat F_t$ satisfies 
    the length-area estimate
    \begin{equation}
    \label{length-area conditions all}
  \frac{\Length_\omega (F_{t}(\D_{R_t}))}{\Area_\omega (F_{t}(\D_{R_t}))}\,
    <\,2^{-t} ,
  \quad
  \frac{\fL_{F_t}(R_t) }{\fA_{F_t}(R_t)}\,
  <\,
  2^{-t}.
    \end{equation}

\smallskip
    \item [\bf ($\heartsuit$).]
    The two normalized currents associated
    with the holomorphic disc $F_t\,\lvert_{\overline \D_{R_t}}$
    are near to the  Nevanlinna/Ahlfors current $\oT_{\mathfrak{m}}$ in the sense that, for all $1\leqslant k,\ell \leqslant n$,
\begin{equation}\label{approximate aimed  Nevanlinna current}
\Big|
\frac{1}{\fA_{F_t}(R_t)}\,
\int_{0}^{R_t}\,
\frac{\dd r}{r}\,
\int_{\D_r}\,
{} F_t^*(\dd z_k \wedge \dd \bar{z}_\ell)\,
-\,
\lp \oT_{\mathfrak{m}},
\dd z_k \wedge \dd \bar{z}_\ell\rp
\Big|
<2^{-t},
 \end{equation}
 \begin{equation}  \label{approximate aimed Ahlfors  current}
\Big|
\frac{1}{\Area_\omega(F_t(\D_{R_t}))} \int_{\D_{R_t}}\, {} F_t^*(\dd z_k \wedge \dd \bar{z}_\ell)\,
-\,
\lp \oT_{\mathfrak{m}},
\dd z_k \wedge \dd\bar{z}_\ell \rp
\Big|\,
<\,
2^{-t}.
\end{equation}

\smallskip
\item [\bf ($\clubsuit$).]

     On $\D_{R_{t-1}+1}$, $\widehat F_t$  is very near to $\widehat F_{t-1}$, i.e., 
    \begin{equation*}
        \label{F_t is small on small disc}\|\widehat F_t (z)-\widehat F_{t-1} (z)\|_{\mathbb{C}^n}< 2^{-t}\cdot \epsilon_{t-1}   \quad\text{for }\, z\in \D_{R_{t-1}+1}.
    \end{equation*}
\end{itemize}

\smallskip

The main technical difficulty in our construction is
to show the existence of 
 $\alpha_t\gg 1$ satisfying the conditions $(\spadesuit), (\heartsuit), (\clubsuit)$.  This will be reached in Proposition~\ref{key technical difficulty} eventually. For the moment, assuming the existence, we can finish the {\bf Step} $t$ by selecting a small positive error bound $\epsilon_t<\epsilon_{t-1}/2$,  such that, if  $\widehat{\mathbf F}$ is some mild holomorphic perturbation of $\widehat F_t$ with $|\widehat{\mathbf F}-\widehat F_t|< \epsilon_{t}$ on $\D_{R_t+1}$, 
 after replacing $F_t$ by the new entire curve   $\mathbf {F}:=\Pi\circ \widehat{\mathbf F}$ on the left-hand-sides and enlarging $2^{-t}$ to $2^{-t+1}$ on the right-hand-sides,
\begin{itemize}
    \smallskip
\item [\bf ($\diamondsuit$).] 
all the corresponding inequalities~\eqref{length-area conditions all},~\eqref{approximate aimed Nevanlinna  current},~\eqref{approximate aimed Ahlfors  current}
maintain. 
\end{itemize}

\smallskip
The existence of such $\epsilon_t$ can be proved by a {\em reductio ad absurdum} argument, since the $\Cc^0$-convergence of a sequence of holomorphic functions
on a larger open set $\mathbb{D}_{R_t+1}\supset \overline{\mathbb{D}}_{R_t}$
can guarantee the $\Cc^1$-convergence on the smaller compact set $\overline{\mathbb{D}}_{R_t}$.

\smallskip

\begin{proof}[Proof of Theorem \ref{thm-class}]
The condition~$(\clubsuit)$ guarantees that  $\lim_{t\rightarrow \infty} \widehat F_t$ exists, which denoted  by $\widehat {\mathbf F}: \mathbb{C}\rightarrow \mathbb{C}^n$.
By our construction, on each $\mathbb{D}_{R_t+1}$ for $t\geqslant 1$, there holds
\[
|\widehat {\mathbf F}-\widehat F_t|\leqslant
\sum_{j\geqslant t}\,
|\widehat F_{j+1}-\widehat F_j|
<
\sum_{j\geqslant t}\,
\frac{\epsilon_j}{2^{j+1}}
<
\sum_{j\geqslant t}\,
\frac{\epsilon_t}{2^{j+1}}
<
\epsilon_t.
\]
Thus the condition~$(\diamondsuit)$
ensures that the entire curve $\mathbf F=\Pi\circ \widehat {\mathbf F} : \mathbb{C}\rightarrow \mathbb{T}$ produces concentric holomorphic discs $\{\mathbf F\,\lvert_{\overline \D_{R_t}}\}_{t\geqslant 1}$ satisfying the length-area conditions for obtaining Nevanlinna/Ahlfors currents, and that for every $t \geqslant 1$, \eqref{approximate aimed Nevanlinna current} and \eqref{approximate aimed   Ahlfors  current} hold after replacing $F_t$ by $\mathbf F$.
Therefore, for any $\oT_{\mathfrak m}$ in~\eqref{fix an enumeration of C_1,n}, the sequence of infinite holomorphic discs
$\{\mathbf F\,\lvert_{\overline \D_{R_t}}\}_{t\in \sigma^{-1}(\mathfrak m)}$
can generate, after passing to certain subsequence, some Nevanlinna and Ahlfors currents 
cohomologous to $\oT_{\mathfrak m}$. 
Lastly, by the  density of
$\{\oT_{\mathfrak m}\}_{\mathfrak m \geqslant  1}$ in  $\mathcal{C}_{1}$
as coholomogy classes in $H^{n-1,n-1}(\T,\R)$, we conclude the proof. 
\end{proof}

The remaining task of this section is  to show 

\begin{proposition}\label{key technical difficulty} 
In each Step $t\geqslant 1$, one can find some large
  integer exponent  $\alpha_t\gg 1$,  such that the requirements  
 $(\spadesuit), (\heartsuit), (\clubsuit)$ hold true.
\end{proposition}

We fix some notation first.
In coordinate systems, write $\widehat F_t=(\widehat f_{t,1},\dots, \widehat f_{t,n})$, 
$\vec v_{t,j} =(v_{t,j,1}, \dots,v_{t,j,n})$
 for every $1\leqslant j\leqslant n$. 
Put
$$\Phi_{t,s}:=\sum_{j=1}^n {\sqrt{\beta_{t, j}}\over \alpha_t}\cdot\varphi_{t, j}  ^{\alpha_t}\cdot v_{t,j,s},
\qquad 1\,\leqslant\, s\,\leqslant \,n,$$
so that $\widehat f_{t,s}=\widehat f_{t-1,s}+\Phi_{t,s}$. 
Set $M:= \max_{1\leqslant j\leqslant n}\max_{\overline \D_{R_t}} |\widehat f'_{t-1,j}|\geqslant 0$, and
$$
m:=\max_{\overline \D( x_{t,1},  1)\cap \overline\D_ {R_t}} |\varphi_{t,1}|>1,
\quad 
m':=\max_{\overline\D_{R_t}}|\varphi'_{t,1}|>0, 
\quad m_0:= \max_{\overline\D_ {R_t}\setminus \D( x_{t,1}, 1)}|\varphi_{t,1}|<1.$$
Here, we omit the indices $t$ on the left-hand-sides for convenience. 
Since every $\varphi_{t,j}$ for $j=2, \dots, n$ is obtained by composing  $\varphi_{t,1}$ with some rotation, we receive
$$
\max_{\overline\D(x_{t,j}, 1)\cap \overline\D_ {R_t}}|\varphi_{t,j}|=m,
\qquad
\max_{\overline\D_{R_t}} |\varphi'_{t,j}|=m', \qquad
\max_{\overline\D_{R_t}\setminus \D(x_{t,j}, 1)}|\varphi_{t,j}|=m_0.
$$ 

\smallskip

\begin{proof}[Proof of the first inequality of~\eqref{length-area conditions all}]
We can compute 
$$\Length_{\omega} \big(F_t( \partial \D_{R_t}) \big) =\int_{\partial \D_{R_t}}
 \big|\widehat F_t^*\omega \big|^{1/2} \quad
\text{and}\quad 
\Area_{\omega}\big(F_t(\D_{R_t})\big) =\int_{\D_{R_t}}\widehat F_t^*\omega, $$
where 
$$\widehat F_t^*\omega=\sum_{s=1}^n {i\over 2} \, \dd \widehat f_{t,s}(z)  \wedge \dd \overline { \widehat f_{t,s}(z)} =  \sum_{s=1}^n  \big|\widehat f_{t,s}'(z)\big|^2
\cdot \frac{i}{2}\, \dd z\wedge \dd \overline z.$$

To get an upper bound for $\Length_{\omega} \big(F_t( \partial \D_{R_t}) \big)$, we need to bound every
\[
|\widehat f_{t,s}'|
=
|\widehat f'_{t-1,s}+\Phi'_{t,s}|
\leqslant M
+
|\Phi'_{t,s}|,
\qquad
 1\leqslant s\leqslant n
\]
 on $\overline \D_{R_t}$. Note that (recall $\norm{v_{t, j}}_{\mathbb{C}^n}=1$)
\begin{equation}\label{computation-Phi-t-s}
|\Phi'_{t,s}| \leqslant \sum_{j=1}^n \sqrt{\beta_{t, j}} \cdot |\varphi_{t,j}|^{\alpha_t-1}\cdot|\varphi'_{t,j}| \cdot |v_{t,j,s}| \leqslant  \sum_{j=1}^n \sqrt{\beta_{t, j}}\cdot  m^{\alpha_t-1}\cdot m'\cdot 1.
\end{equation}
Hence we receive 
\begin{equation} 
\label{length-upper-bound-class}
\Length_{\omega} \big(F_t( \partial \D_{R_t}) \big)
<
\int_{\partial \D_{R_t}}
\Big(
M+
\sum_{j=1}^n  \sqrt{\beta_{t, j}}\cdot m^{\alpha_t-1}m' 
\Big)\, |\dd z| 
\lesssim 
R_t \cdot  m^{\alpha_t-1} m'.
\end{equation}
Here, the implicit multiplicative  does not depend on $\alpha_t$.

Now we estimate  $\Area_{\omega}\big(F_t(\D_{R_t})\big)$. In each non-empty open set 
$$\D_{R_t} \cap \big\{1< m^{3/4} < |\varphi_{t, j}| <m   \big\}\cap \big\{\varphi'_{t, j}\neq 0 \big\},
\qquad
j= 1,\,\dots,\, n,$$
we 
take a closed ball $B_{t, j}$. Clearly, $B_{t, j}$ is contained in $\D(x_{t, j}, 1)$. By the compactness of $B_{t, j}$, we can find  some $c_{t,  j}>0$ such that $|\varphi'_{t, j}|\geqslant c_{t, j}$ on $B_{t, j}$. By triangle inequality, for every $z\in B_{t, j}$, for sufficiently large $\alpha_t\gg 1$, we have
\begin{align}
\nonumber
|\Phi'_{t,s}(z)|
&\geqslant
\sqrt {\beta_{t, j}}\cdot
|\varphi_{t, j} | ^{\alpha_t-1}  |\varphi'_{t, j}| \cdot |v_{t,j,s}|-  \sum_{k=1, k\neq j}^{n}  \Big(
\sqrt {\beta_{t, k}}\cdot
|\varphi_{t, k} | ^{\alpha_t-1} |\varphi'_{t, k}| \cdot |v_{t, k, s}| \Big)
\\
\label{first time estimate of Phi}
&\geqslant
\sqrt {\beta_{t, j}}\cdot\,
( m^{3/4} )^{\alpha_t-1}\, c_{t, j} \cdot |v_{t,j,s}|  -  1,
\end{align}
since $|\varphi_{t,k}|\leqslant m_0<1$ on $B_{t, j}\subset\D(x_{t,j}, 1)\subset \overline\D_{R_t}\setminus \D(x_{t,k}, 1)$ for $k\neq j$.

Recall that $\sum_{j=1}^n \beta_{t, j}=1$ and that $\sum_{s=1}^n |v_{t,j,s}|^2=1$ for every $j=1, \dots, n$. We can take  some $j_0$ and $s_0$  with ${\beta_{t, j_0}}\geqslant 1/n$ and $|v_{t,j_0, s_0}|\geqslant 1/\sqrt n$. Hence for $z\in B_{t, j_0}$ and for  $\alpha_t \gg 1$,
we can continue to estimate~\eqref{first time estimate of Phi}
as 
$
|\Phi'_{t,s_0}(z)|
\geqslant
(2n)^{-1} ( m^{3/4} )^{\alpha_t-1}\, c_{t, j_0}
$
since $m>1$.  Thus we get
$$|\widehat f'_{t, s_0}|
=
|\widehat f'_{t-1,s_0}+\Phi'_{t,s_0}|
\geqslant |\Phi'_{t, s_0}| -M\geqslant (4n)^{-1} ( m^{3/4} )^{\alpha_t-1}\, c_{t, j_0}$$
on $B_{t,j_0}$, again assuming  $\alpha_t \gg 1$. 
Consequently,
\begin{equation}\label{area-lower-bound-class}
\Area_{\omega}\big(F_t(\D_{R_t})\big)\geqslant \Area_{\omega}\big(F_t(B_{t,j_0})\big) \geqslant 
\int_{B_{t,j_0}} \big| \widehat f'_{t,s_0}(z) \big|^2\, \frac{i}{2}\,\dd z\wedge \dd \overline z \gtrsim (m^{3/2} )^{\alpha_t-1}.
\end{equation}
Here, the implicit multiplicative depends on $B_{t, j_0}$, and does not depend on $\alpha_t$.

By~\eqref{length-upper-bound-class} and~\eqref{area-lower-bound-class}, we conclude that 
 $$ \frac{\Length_\omega (F_{t}(\D_{R_t}))}{\Area_\omega (F_{t}(\D_{R_t}))} 
 \lesssim
 \frac{R_t \cdot  m^{\alpha_t-1} m'}{(m^{3/2} )^{\alpha_t-1}}
=
 \frac{R_t\cdot m'}{m^{\alpha_t/2 -1/2}}.  $$
 Letting $\alpha_t\gg 1$, we finish the proof.
\end{proof}

\smallskip

\begin{proof}[Proof of the second inequality of \eqref{length-area conditions all}]
Recall the definition
$$
\fL_{F_t}(R_t) := \int_0^{R_t}\Length_{\omega} \big( F_t(\partial \D_r)\big)\,\frac{\dd r}{ r}\quad\text{and}\quad  
\fA_{F_t}(R_t) := \int_0 ^{R_t}  \Area_\omega \big(F_{t}(\D_r)\big)
\,\frac{\dd r}{ r}.$$ 
By the same argument as~\eqref{length-upper-bound-class}, for every $r\leqslant R_t$, we can show that
\[
\Length_{\omega} \big( F_t(\partial \D_r)\big) \lesssim r\cdot m^{\alpha_t -1} m',
\]
where the implicit multiplicative is independent of  $\alpha_t$.
Hence
$$\fL_{F_t}(R_t) \lesssim   \int_0^{R_t} r\cdot m^{\alpha_t -1} m'\,\frac{\dd r}{ r} \leqslant  R_t\cdot m^{\alpha_t -1} m'.  $$

To bound $\fA_{F_t}(R_t)$ from below, we consider 
$R_t^-:=\max\big\{ |z|:\,  z\in B_{t,j_0} \big\}<R_t$.  
For all $ r\in [R_t^-, R_t]$, \eqref{area-lower-bound-class} gives 
$ \Area_\omega \big(F_{t}(\D_r)\big)\gtrsim (m^{3/2})^{\alpha_t -1}$.
Hence 
$$\fA_{F_t}(R_t) \geqslant \int_{R_t^-}^{R_t}  \Area_\omega \big(F_{t}(\D_r)\big)
\,\frac{\dd r}{ r} \gtrsim \int_{R_t^-}^{R_t} (m^{3/2})^{\alpha_t -1} \,\frac{\dd r}{ r} \gtrsim (m^{3/2})^{\alpha_t -1},$$
where the implicit multiplicative is independent of $\alpha_t$.

Thus
we conclude that 
$$\frac{\fL_{F_t}(R_t) }{\fA_{F_t}(R_t)}
\lesssim
 \frac{R_t \cdot  m^{\alpha_t-1} m'}{(m^{3/2} )^{\alpha_t-1}}
=
 \frac{R_t\cdot m'}{m^{\alpha_t/2 -1/2}}.  $$
 By taking sufficiently large  $\alpha_t\gg 1$, we finish the proof.
\end{proof}

\smallskip

\begin{proof}[Proof of \eqref{approximate aimed Ahlfors  current}]
Fix $k, \ell\in \{1, \dots, n\}$.
By \eqref{value-T-v}, we have
$$\lp \oT_{\mathfrak{m}},
\dd z_k \wedge \dd\bar{z}_\ell \rp= \sum_{j=1}^n \beta_{t, j}\cdot \lp \oT_{[\vec v_{t,j}]},  \dd z_k \wedge \dd\bar{z}_\ell\rp    =  -2i \sum_{j=1}^n  \beta_{t, j}\cdot 
v_{t,j,k}\overline v_{t,j ,\ell}.$$ 
Now we compute
$$\widehat F_t^*(\dd z_k \wedge \dd \overline z_\ell)=\dd \widehat f_{t,k}(z) \wedge \dd \overline {\widehat f_{t,\ell} (z) }    =\widehat f_{t,k}'(z)\overline {\widehat f_{t,\ell}'(z)}   \, \dd z \wedge \dd \overline z.$$
Decompose
$$\widehat f_{t,k}'(z)\cdot \overline {\widehat f_{t,\ell}'(z)}
=
(\widehat f'_{t-1, k}+\Phi'_{t, k})
\cdot \overline{
(\widehat f'_{t-1, \ell}+\Phi'_{t, \ell})
}
=
\Theta_1(z) +\Theta_2(z)+\Theta_3(z),$$
where the dominant term is
\begin{align*}
\Theta_1(z)
:=
\Phi'_{t,k}(z) \cdot \overline {\Phi'_{t,\ell}(z)}=
\Big( \sum_{j=1}^n \sqrt{\beta_{t, j}}\cdot\varphi_{t, j}  ^{\alpha_t-1}  \varphi'_{t, j}   \cdot v_{t,j,k} \Big) \cdot \Big( \sum_{s=1}^n \sqrt{\beta_{t, s}}\cdot \overline \varphi_{t, s}  ^{\alpha_t-1} \overline{\varphi'_{t, s}}   \cdot \overline v_{t,s,\ell}   \Big),
\end{align*}
and where the remaining terms are
$$\Theta_2(z):=\widehat f_{t-1,k}'(z)\cdot \overline{\widehat f'_{t-1,\ell}(z)},\quad \Theta_3(z):=\widehat f_{t-1,k}'(z)\cdot \overline {\Phi'_{t,\ell}(z)}+ \Phi'_{t,k}(z) \cdot \overline{\widehat f'_{t-1,\ell}(z)}.$$
By our construction, if $j\neq s$, then $|\varphi_{t,j}|^{\alpha_t-1}|\varphi_{t,s}|^{\alpha_t-1}\leqslant m^{\alpha_t-1}$ on $\D_{R_t}$. Noting that $|\beta_{t,j}|\leqslant 1, |v_{t,j,k}|\leqslant 1$, we have
$$ \Theta_1 =   \sum_{j=1}^n \beta_{t, j} \cdot v_{t,j,k} \overline v_{t,j,\ell}\cdot |\varphi_{t,j}|^{2\alpha_t -2}|\varphi_{t,j}'|^2  
+
O(m^{\alpha_t -1}m'^2).$$
By the same reasoning,
we have
$
|\Theta_2|=O(M^2),
|\Theta_3|= O(M m^{\alpha_t -1}m').
$
Denote
$$\Ac_{t,j}:=\int_{ \D_{R_t} }  |\varphi_{t,j}|^{2\alpha_t -2}|\varphi_{t,j}'|^2 (z) \, \dd z \wedge \dd \bar z ,
\qquad
j= 1,\,\dots,\, n.
$$
Observe that each $\Ac_{t,j}=\Ac_{t,1}$, 
since $\varphi_{t, j}$
is obtained by composing $\varphi_{t, 1}$ with some rotation.

Recall~\eqref{area-lower-bound-class} that $ \Area_\omega \big(F_{t}(\D_{R_t})\big)\gtrsim (m^{3/2})^{\alpha_t -1}$. By comparing the exponents of $m$, we obtain that
\[
\frac{1}{\Area_\omega(F_t(\D_{R_t}))}  \Big|
\int_{\D_{R_t}}\, \widehat F_t^*(\dd z_k \wedge \dd \bar{z}_\ell)   - \sum_{j=1}^n \beta_{t, j} \cdot v_{t,j,k} \cdot \overline v_{t,j,\ell}\cdot    \Ac_{t,j} \Big|
\lesssim \,
\frac{ m^{\alpha_t -1}}{(m^{3/2})^{\alpha_t -1}},
\]
where the implicit multiplicative is independent of $\alpha_t$.
  Using that all $\Ac_{t,j}$'s are equal, we receive that
\begin{equation}\label{F-t-A-t-1}
\frac{1}{\Area_\omega(F_t(\D_{R_t}))}  \Big|
\int_{\D_{R_t}}\, \widehat F_t^*(\dd z_k \wedge \dd \bar{z}_\ell)   - \Ac_{t,1} 
 \sum_{j=1}^n \beta_{t, j} \cdot v_{t,j,k}\cdot \overline v_{t,j,\ell}   \Big|\longrightarrow 0 
 \end{equation}
as $\alpha_t$ tends to infinity. Hence to prove \eqref{approximate aimed Ahlfors  current}, it remains to show  that 
\begin{equation}\label{A-t-1-Area}
\Ac_{t,1} /\Area_\omega\big(F_t(\D_{R_t})\big)=-2i +o(1) \quad\text{as} \quad \alpha_t\to \infty.
\end{equation}

By letting $k=\ell$ running through $1, \dots, n$, we deduce from \eqref{F-t-A-t-1} that
$$\frac{1}{\Area_\omega(F_t(\D_{R_t}))}  \Big|
\int_{\D_{R_t}}\, \widehat F_t^*\Big(  \frac{i}{2}\sum_{j=1}^n \dd z_j\wedge \dd \bar{z}_j  \Big)  - \frac{i}{2} \Ac_{t,1} 
 \sum_{j=1}^n  \beta_{t, j} \cdot \sum_{k=1}^n |v_{t,j,k}|^2   \Big|\longrightarrow 0.   $$
 Equivalently, 
$$\frac{1}{\Area_\omega(F_t(\D_{R_t}))}  \Big|
 \Area_\omega\big(F_t(\D_{R_t})\big)  - \frac{i}{2}\Ac_{t,1} 
 \sum_{j=1}^n  \beta_{t, j} \cdot  \norm{\vec v_{t,j}}_{\C^n}^2   \Big|\longrightarrow 0.   $$
Thus~\eqref{A-t-1-Area} follows from $\sum_{j=1}^n \beta_{t, j}=1$ and $\norm{\vec v_{t,j}}_{\C^n}=1$. This concludes the proof.
\end{proof}

\smallskip

\begin{proof} [Proof of \eqref{approximate aimed  Nevanlinna current}]
For every $j=1,\dots,n$, we define 
$$\Bc_{t,j}:= \int_{0}^{R_t}\,
\frac{\dd r}{r}\,
\int_{\D_r}\,
|\varphi_{t,j}|^{2\alpha_t -2}|\varphi_{t,j}'|^2 \,\dd z \wedge \dd \bar{z}= \int_{0}^{R_t}\,
\frac{\dd r}{r}\,
\int_{\D_r}\,
|\varphi_{t,1}|^{2\alpha_t -2}|\varphi_{t,1}'|^2 \,\dd z \wedge \dd \bar{z}. $$
Repeating the same argument, by comparing the exponents of $m$, for $\alpha_t \to \infty$,  we have
$$
\frac{1}{\fA_{F_t}(R_t)}\Big| 
\int_{0}^{R_t}\,
\frac{\dd r}{r}\,
\int_{\D_r}\,
\widehat F_t^*(\dd z_k \wedge \dd \bar{z}_\ell)\,
-\,
\Bc_{t,1} 
 \sum_{j=1}^n \beta_{t, j} \cdot v_{t,j,k} \cdot \overline v_{t,j,\ell}  
\Big|
\longrightarrow 0. 
$$
Now we only need to show that $\Bc_{t,1}/\fA_{F_t}(R_t)=-2i+o(1)$ as $\alpha_t\to \infty$. This  can be done by the same argument as  in the preceding proof. 
\end{proof}

\smallskip

\begin{proof}[Proof of Proposition~\ref{key technical difficulty}]
Summarizing,  $(\spadesuit)$ and $(\heartsuit)$ 
are proved already. 
The requirement
$(\clubsuit)$ is easily satisfied, since
 on $\D_{R_{t-1}+1}$, by letting $\alpha_t$ tend to infinity, we have
    \[
    \|\widehat F_t-\widehat F_{t-1}\|_{\mathbb{C}^n}
    =
    \Big\|\sum_{j=1}^{n}\,{\sqrt{\beta_{t, j}} \over \alpha_t}\cdot\varphi_{t, j}  ^{\alpha_t}\cdot
\vec{v}_{t, j} \Big\|_{\mathbb{C}^n}\,
\longrightarrow \,
0
\]
as  $|\varphi_{t, j}|\leqslant m_0<1$.
\end{proof}

\medskip

\section{\bf Exotic entire curves producing all shapes of Nevanlinna currents} \label{Sect. 3}

\subsection{An algorithm for constructing exotic entire curves}
Let  $\pi: \C\rightarrow E=\C/\Gamma$ be the canonical projection. 
 The \textit{Weierstrass $\wp$-function} associated with $\Gamma$ is given by 
 $$\wp (z):= {1\over z^2} +\sum_{\lambda \in \Gamma \setminus \{0\}} \Big( {1\over (z-\lambda)^2}-{1\over \lambda ^2}  \Big).$$
It is  meromorphic on $\C$, doubly periodic with respect to $\Gamma$. Hence it induces a holomorphic map, still denoted by $\wp$, from $\C$ to $\C\P^1$, and a holomorphic map $\widetilde\wp: E\rightarrow \C\P^1$ with $\wp=\widetilde\wp\circ\pi$.
Let $\Pi:=(\wp, \pi): \mathbb{C}^2\rightarrow X=\C\P^1\times E$ be the total projection.
Our desired entire curve $\mathbf F$
for Theorem~\ref{thm-shape} 
will be obtained as $\mathbf F:=
\Pi\circ \widehat{\mathbf F}$, where $\widehat{\mathbf F}$ is the limit of some sequence of entire curves
$\{\widehat F_t: \mathbb{C}\rightarrow \C^2\}_{t\geqslant 1}$. The concentric holomorphic discs
$\{\mathbf  F\,\lvert_{\overline \D_R } \}_{R>0}$ will produce all shapes of Nevanlinna/Ahlfors currents on $X$. Fix a reference Hermitian form  $\omega:=\pi_1^*\omega_\FS+ \pi_2^*\omega_E$ on $X$.

\medskip

  Take a dense   sequence  $\{x_j\}_{j=1}^\infty$
  of distinct points  in $\C\P^1$, and another dense 
 sequence $\{y_k\}_{k=1}^\infty$ 
  of distinct points
  in $E$.
\smallskip
Enumerate the countable set
\[
\Bigg\{ \Big[(P, Q); (A_j)_{j=1}^P; (B_k)_{k=1}^Q \Big] : \, (P, Q)\in \Z_{\geqslant 0}^2\setminus \{(0, 0)\}, A_{j}, B_{k}\in \mathbb{Q}_+, \sum_{j=1}^P A_j+\sum_{k=1}^Q B_k=1  \Bigg\}
\]
as $\big\{ \big[(P_{\mathfrak m}, Q_{\mathfrak m}); \mathbf A_\fm; \mathbf B_\fm \big] \big\}_{\mathfrak m \geqslant 1}$. 
   The value  $P_{\mathfrak m}$ (resp.\ $Q_{\mathfrak m}$)
  counts the number of irreducible elliptic (resp.\  rational) components in the singular part of certain Nevanlinna/Ahlfors current $\oT$ to be obtained, 
  while  $\bA _\fm\in \mathbb{Q}_+^{P_{\mathfrak m}}$ (resp.\ $\bB_\fm\in \mathbb{Q}_+^{Q_{\mathfrak m}}$) records the  masses of $\oT$ on these elliptic (resp.\ rational) curves.  

\smallskip

As in Sect.~\ref{sect. 2}, we  fix a bijection 
 \begin{equation*}\label{defn-sigma}
  (\sigma,\tau) : \mathbb{Z}_{+}\longrightarrow \mathbb{Z}_{+}\times \mathbb{Z}_{+}.
\end{equation*}
  
\medskip

\noindent \textbf{Step $0$:}
Set $R_0:=2, \ep_0:=1$. Let $\widehat F_0: \C\rightarrow \C^2$ be  any  entire curve.

\medskip 

\noindent In {\bf Step $t-1$} for $t\geqslant 1$, we record an entire curve $\widehat F_{t-1}=(\widehat g_{t-1},\widehat h_{t-1})$ in $\C^2$, some large radius $R_{t-1}\gg 1$,  and a small error bound $0<\ep_{t-1}\ll 1$.

\medskip 

\noindent \textbf{Step $t$:}  Read $\mathfrak{m}=\sigma(t)$. 
Set $$R_t:=\max\{2R_{t-1}, P_{\mathfrak{m}}+Q_{\mathfrak{m}}\}.$$  
We choose $P_{\mathfrak{m}}+Q_{\mathfrak{m}}$ distinct points $\{\tx_{t,j}\}_{j=1}^{P_{\mathfrak{m}}}$, $\{\ty_{t,k}\}_{k=1}^{Q_{\mathfrak{m}}}$ in  $\partial\D_{R_t}$, 
such that the closed unit discs
$\{\overline{\mathbb{D}}(\tx_{t,j}, 1)\}_{j=1}^{P_{\mathfrak{m}}}$, $\{\overline{\mathbb{D}}(\ty_{t,k}, 1)\}_{k=1}^{Q_{\mathfrak{m}}}$ 
are pairwise disjoint.
Without loss of generality, we may assume that $P_{\mathfrak{m}}\geqslant 1$. 
By Lemma \ref{lem-exist}, for the point $\tx_{t, 1}$ and its neighborhood $\D(\tx_{t, 1}, 1)$, we can find some entire function $\varphi_{t, 1}$ such that 
\begin{equation}\label{condition-varphi}
|\varphi_{t, 1}(\tx_{t, 1})|>1  \quad\text{and}\quad |\varphi_{t, 1}|<1 \quad \text{on }\, \, \, \overline \D_{R_{t}}\setminus \D(\tx_{t, 1}, 1).
\end{equation}
Since  $\tx_{t,1}, \tx_{t,2},\dots,\tx_{t, {P_{\mathfrak{m}}}}$ are in the same circle centered at the origin, 
by composing $\varphi_{t, 1}$ with some rotations,
 we can also define $\varphi_{t,j}(z):=  \varphi_{t,1} (e^{i\theta_{t,j}}\cdot z)$ with likewise property, 
where each angle $\theta_{t,j}\in [0, 2\pi)$ is uniquely determined by
$e^{i\theta_{t,j}}\cdot \tx_{t, j}=\tx_{t, 1}$
for $j=2, \dots, P_{\mathfrak{m}}$.
Similarly, we can obtain $\psi_{t,k}$ for $1\leqslant k\leqslant Q_{\mathfrak{m}}$  by composing $\varphi_{t,1}$ with some rotations, such that 
\begin{equation}\label{condition-psi}
|\psi_{t, k}(\ty_{t, k})|>1 , \quad |\psi_{t, k}|<1 \quad \text{on }\, \, \,\overline \D_{R_{t}}\setminus \D(\ty_{t, k}, 1).
\end{equation}

\smallskip
Let $\delta_t\in (0, 2^{-t})$ be a  small constant satisfying 
\begin{equation*}\label{defn-delta-t}
\dist_{\omega_\FS}(x_{j_1},x_{j_2})>6\delta_t, \quad \dist_{\omega_E}(y_{k_1},y_{k_2})>6\delta_t,
\end{equation*}
for all $1 \leqslant j_1\neq j_2 \leqslant P_{\mathfrak{m}}, 1 \leqslant  k_1\neq k_2 \leqslant Q_{\mathfrak{m}}$.

 Since $\wp : \C \rightarrow \C\P^1$ and $\pi: \C\rightarrow E$ are Lipschitz and surjective, 
 applying  Lemma~\ref{lem-disc-wp-pi}, 
 we can find auxiliary entire functions $g_{t-1}^\star,h_{t-1}^\star$, such that on $\overline\D_{R_{t-1}+1}$
 \begin{equation}
     \label{g* is very near to}
     |\widehat{g}_{t-1}-g^\star_{t-1}|<\epsilon_{t-1}\cdot 2^{-t-2},\qquad
     |\widehat{h}_{t-1}-h^\star_{t-1}|<\epsilon_{t-1}\cdot 2^{-t-2}, 
 \end{equation}
$$ \dist_{\omega_\FS}\big(\wp\circ \widehat g_{t-1} ,\wp\circ g_{t-1}^\star   \big)<\delta_t, \qquad \dist_{\omega_E}\big(\pi\circ \widehat h_{t-1} ,\pi\circ h_{t-1}^\star   \big)<\delta_t,$$
and such that for every $1\leqslant j\leqslant P_{\mathfrak{m}}, 1\leqslant k\leqslant Q_{\mathfrak{m}}$,
\begin{equation}\label{dist-x-h}
\dist_{\omega_\FS}\big(x_j ,\wp\circ g_{t-1}^\star\big)<\delta_t \quad\text{on }\,\, \,  \overline\D (\tx_{t,j},1),    
\end{equation}
\begin{equation}\label{dist-y-g}
\dist_{\omega_E}\big(y_k ,\pi\circ h_{t-1}^\star\big)<\delta_t \quad\text{on }\, \, \overline\D (\ty_{t,k},1).
\end{equation}

Now we define 
\begin{equation}\label{defn-g-t-case1}
\widehat g_{t}:=  g_{t-1}^\star  + \sum_{k=1}^{{Q_{\mathfrak{m}}}}  q_{t, k}\cdot\psi_{t,k}  ^{\alpha_t} =: g_{t-1}^\star + \Psi_t ,
\end{equation}
\begin{equation}\label{defn-h-t-case1}
\widehat h_{t}:=h_{t-1}^\star + \sum_{j=1}^{{P_{\mathfrak{m}}}}  p_{t, j}\cdot\varphi_{t,j}  ^{\alpha_{t}}  =: h_{t-1}^\star  +\Phi_t,
\end{equation}
where the parameters $p_{t, j}, q_{t,k}\in (0, 1]$ and  the exponent $\alpha_{t}\gg 1$ are to be determined later by sophisticated reasoning.
Thus
we obtain an entire curve $\widehat F_{t}:=(\widehat g_{t},\widehat h_{t})$ in $\C^2$. 

\smallskip
The goal of {\bf Step $t$} is to select some appropriate parameters $\alpha_{t},
p_{t, \bullet},
q_{t, \bullet}$,
so that the obtained entire curve $\widehat{F}_t$
satisfies  the following  Propositions \ref{key-prop-shape} and \ref{Area distribution},
whose proofs will be postponed to Sect.~\ref{subsection: proof of  lemma 3.3}. 
Define the following open subsets 
$$U_t :=  \bigcup_{j=1}^{P_{\mathfrak{m}} } \D(x_{j},2\delta_t)\times E ,\quad V_t:=  \C\P^1 \times \bigcup_{k=1}^{Q_{\mathfrak{m}} } \D(y_{k},2\delta_t)$$ 
of $\C\P^1 \times E$, 
and the following open subsets
$$\widehat U_t :=  \bigcup_{j=1}^{P_{\mathfrak{m}} } \D(\tx_{t,j}, 1),\quad \widehat V_t:=   \bigcup_{k=1}^{Q_{\mathfrak{m}} } \D(\ty_{t,k}, 1)$$
 of $\C$.
For every possible $ j,k,\ell$, we put
\begin{align*}
\Ac_{t,j}
&
:=   \Area_{\omega} \big(F_t \big(\D(\widehat x_{t,j}, 1)\cap  \D_{R_t}\big)\big),
\quad
\Ac_{t,j}^\oN
:=
\int_0^{R_t}    \Area_{\omega} \big(F_t \big(\D(\widehat x_{t,j}, 1)\cap  \D_r\big)\big) \,{\dd r\over r};
\\
\Bc_{t,k}
&
:=   \Area_{\omega} \big(F_t\big(\D(\widehat y_{t,k}, 1) \cap \D_{R_t}\big)\big),
\quad
\Bc_{t,k}^\oN
:= \int_0^{R_t}   \Area_{\omega} \big(F_t\big(\D(\widehat y_{t,k}, 1) \cap \D_r\big)\big) \,{\dd r\over r}.
\end{align*}

\begin{proposition}
  \label{key-prop-shape}
In every Step $t\geqslant 1$, there exists certain large integer $M_t\gg 1$  such that, if $\alpha_t\geqslant M_t$ and
$\max_{j}\{ p_{t,j}\cdot\alpha_t\}\geqslant 1,
\, \max_{k}\{q_{t,k}\cdot\alpha_t\} \geqslant 1
$, then the following estimates hold simultaneously.
\begin{enumerate}
\item [\bf ($\clubsuit$).]
     On $ \D_{R_{t-1}+1}$, $\widehat F_t$ is very close to $\widehat F_{t-1}$, more precisely 
    \begin{equation}
        \label{F_t is small on small disc-shape}
        \|\widehat F_t(z)-\widehat F_{t-1} (z)\|_{\mathbb{C}^2}< 2^{-t}\cdot \epsilon_{t-1}   \quad\text{for }\, z\in \D_{R_{t-1}+1}.
    \end{equation}
    
    \smallskip
    
 \item [\bf ($\spadesuit$).]
    The entire curve $F_t:=\Pi\circ \widehat F_t$ satisfies
    the length-area conditions 
    \begin{equation}
    \label{length-area conditions all-shape}
  \frac{\Length_\omega (F_{t}(\D_{R_t}))}{\Area_\omega (F_{t}(\D_{R_t}))}\,
    <\,2^{-t} ,
  \quad
  \frac{\fL_{F_t}(R_t) }{\fA_{F_t}(R_t)}\,
  <\,
  2^{-t}.
    \end{equation}
    
   \smallskip
   
\item[($\heartsuit$).] One has
\begin{equation}
    \label{I=0, Ahfors diffuse is trivial, estimate}
 { \Area_{\omega} (F_t(\D_{R_t})\setminus (U_t \cup V_t) ) \over \Area_{\omega}(F_t(\D_{R_t}))  } < 2^{-t},  
 \end{equation} 
\begin{equation}
    \label{I=0, Nevanlinna diffuse is trivial, estimate} 
    {1\over \fA_{F_t}(R_t)   }  \int_{0} ^{R_t}  
\frac{\dd r}{ r} \Area_{\omega}\big(F_t(\D_r)\setminus (U_t \cup V_t)\big)<2^{-t}.  
\end{equation} 
\end{enumerate}
\end{proposition}

\smallskip

\begin{proposition}
\label{Area distribution}
Let $M_t$ be as in Proposition \ref{key-prop-shape}.
In every Step $t\geqslant 1$, given
\[
\Omega_0:=
[A_1 : \cdots:
A_{P_{\mathfrak{m}}}
:
B_1
:
\cdots
:
B_{Q_{\mathfrak{m}}}]
\in
\mathbb{RP}^{P_{\mathfrak{m}}+Q_{\mathfrak{m}}-1}  
\]
with all coordinates positive.
There exist $\alpha_t, p_{t,\bullet}, q_{t,\bullet}$ satisfying $\alpha_t\geqslant M_t$ and
$\max_{j}\{ p_{t,j}\cdot\alpha_t\}\geqslant 1,
\, \max_{k}\{q_{t,k}\cdot\alpha_t\} \geqslant 1$, such that,  on $\R\P^{P_{\mathfrak{m}}+Q_{\mathfrak{m}}-1}$ with a distance $\dist$, one has
$$\dist \big(  [\Ac_{t,1}:\cdots: \Ac_{t,P_{\mathfrak{m}}} : \Bc_{t,1}:\cdots: \Bc_{t,Q_{\mathfrak{m}}} ]       ,         \, \Omega_0 \big)
<2^{-t}, $$
$$ \dist \big(  [\Ac_{t,1}^\oN:\cdots: \Ac_{t,P_{\mathfrak{m}}}^\oN : \Bc_{t,1}^\oN:\cdots: \Bc_{t,Q_{\mathfrak{m}}}^\oN ]       ,         \, \Omega_0 \big)<2^{-t}. $$ 
\end{proposition}

\smallskip

 Applying Proposition \ref{Area distribution} with  $\Omega_0=[\bA_\fm : \bB_\fm]$,
we can obtain some admissible parameters $\alpha_t, p_{t,\bullet}, q_{t,\bullet}$.

\smallskip

Lastly, we finish the {\bf Step} $t$ by selecting a small positive error bound 
 $\epsilon_t\ll \epsilon_{t-1}/2$,
 such that, for any mild holomorphic perturbation $\widehat{\mathbf F}$ of $\widehat F_t$ with $\|\widehat{\mathbf F}-\widehat F_t\|_{\C^2}< \epsilon_{t}$ on $\D_{R_t+1}$, 
 after replacing $F_t$ by  $\mathbf {F}:=\Pi\circ \widehat{\mathbf F}$ and enlarging $2^{-t}$ to $2^{-t+1}$ on the right-hand-sides,
\begin{itemize}
    \smallskip
\item [\bf ($\diamondsuit$).] 
\text{all the corresponding estimates in Propositions \ref{key-prop-shape} and \ref{Area distribution} maintain.}
\end{itemize}

The existence of such $\epsilon_t$ in the
  consideration
($\diamondsuit$)
can be proved by a {\em reductio ad absurdum} argument,  since the $\Cc^0$-convergence of a sequence of holomorphic functions
on a larger open set $\D_{R_t+1}$
can guarantee the $\Cc^1$-convergence on the smaller compact set $\overline{\D}_{R_t}$ by Cauchy's integral formula.

\medskip

Let us we briefly explain  the motivation.  
($\clubsuit$) ensures that  $\{\widehat F_t\}_{t\geqslant 1}$ converges to an entire curve $\widehat{\mathbf F}: \mathbb{C}\rightarrow \C^2$. ($\spadesuit$) can guarantee that
$\mathbf F=
\Pi\circ \widehat{\mathbf F}$ also satisfies 
a similar  length-area condition, provided that $\{\epsilon_t\}_{t\geqslant 1}$ shrink to zero sufficiently fast.
Similarly, ($\heartsuit$) is  to force $\mathbf F$ to produce Nevanlinna/Ahlfors currents having singular parts supported
in precise locations.

Now we sketch our idea of generating singular parts of Nevanlinna/Ahlfors currents. 
For any $\mathfrak{m}\geqslant 1$, write $\sigma^{-1}(\mathfrak{m})$ in the increasing order as
$\{t_{\mathfrak m, s}\}_{s\geqslant 1}$. 
By Propositions \ref{key-prop-shape} and \ref{Area distribution},  
we can check that any Nevanlinna/Ahlfors current $\mathcal{T}$ produced by some subsequence of  $\{\mathbf F \,\lvert_{\overline \D_{R_{t_{\mathfrak m, s}}}}\}_{s \geqslant 1}$ has zero mass outside $ \bigcup_{j=1}^{P_{\mathfrak m}} \{x_{j} \} \times E$ and $\C\P^1 \times \bigcup _{k=1}^{Q_{\mathfrak m}} \{y_k\}$.
Moreover, 
we can make sure that $\mathcal{T}=\mathcal{T}_{\alg}
$  charges  preassigned positive masses $\bA_\fm$ 
and  $\bB_\fm$ 
on each of
these irreducible algebraic curves.

\subsection{Proofs of Propositions \ref{key-prop-shape} and \ref{Area distribution}}\label{subsection: proof of  lemma 3.3}

\begin{proof}[Proof of~\eqref{F_t is small on small disc-shape}]
Note that on $\D_{R_{t-1}+1}$, all $\varphi_{t, \bullet}, \psi_{t, \bullet}$ 
have moduli strictly less than $1$,  and $p_{t,\bullet},q_{t,\bullet}\in (0,1]$ by our construction.
Thus by letting  $\alpha_t\geqslant M_t\gg 1$, 
 from \eqref{defn-g-t-case1}, \eqref{defn-h-t-case1},  on $\D_{R_{t-1}+1}$ we have
\[
|
\widehat{g}_t
-g_{t-1}^\star
|
<
\epsilon_{t-1}\cdot 2^{-t-2},
\qquad
|
\widehat{h}_t
-h_{t-1}^\star
|
<
\epsilon_{t-1}\cdot 2^{-t-2}.
\]
This combining with~\eqref{g* is very near to} finishes the proof.
\end{proof}

To prove the condition
 $(\spadesuit)$,
we introduce the following notation 
$$M:= \max_{\overline \D_{R_t}} \big\{ |(g_{t-1}^\star
)'|,|(h_{t-1}^\star)'| \big\}, \qquad m:=\max_{\overline \D_{R_t}\cap\overline \D(\widehat x_{t,1},  1)} |\varphi_{t,1}|>1,$$
$$m':=\max_{\overline\D_ {R_t}}|\varphi'_{t,1}|>0, \qquad m_0:= \max_{\overline\D_ {R_t}\setminus \D(\widehat x_{t,1}, 1)}|\varphi_{t,1}|<1.$$
Since every $\varphi_{t,j}$ for $j=2, \dots, P_{\mathfrak{m}}$ is obtained by composing  $\varphi_{t,1}$ with some rotation, we have 
$$
m=\max_{\overline \D_{R_t}\cap\overline\D(\widehat x_{t,j}, 1)}|\varphi_{t,j}|,
\qquad
m'=\max_{\overline\D_{R_t}} |\varphi'_{t,j}|, \qquad
m_0=\max_{\overline\D_{R_t}\setminus \D(\widehat x_{t,j}, 1)}|\varphi_{t,j}|.
$$ 
Similar equalities also hold  for the functions $\psi_{t, \bullet}$.

\medskip

\noindent{\bf Convention.} From now on, we use $O_t(1)$ to denote some positive constants 
independent of $\alpha_t$ (but can 
depend on $t, R_t,  \varphi_{t, 1}$, etc), which vary according to the context.

\smallskip

\begin{proof}[Proof of the first inequality of \eqref{length-area conditions all-shape}]
Write $\wp^*\omega_\FS=\vartheta(z)\, i \dd z\wedge\dd \overline z$ for some  smooth function  $\vartheta\geqslant 0$ on $\C$. Since $\wp$ is doubly periodic with respect to the lattice
$\Gamma$, so is $\vartheta$. In particular,  $\vartheta$ is  bounded from above. 
By definition,
$$\Length_{\omega} \big(F_t( \partial \D_{R_t}) \big) =\int_{\partial \D_{R_t}}
 \big|\widehat F_t^*\omega \big|^{1/2} \quad
\text{and}\quad 
\Area_{\omega}\big(F_t(\D_{R_t})\big) =\int_{\D_{R_t}}\widehat F_t^*\omega, $$
where 
\begin{align}
\widehat F_t^*\omega
&= 
\vartheta \big(\tg_t(z)\big)\, i\, \dd \tg_t(z) \wedge \dd \overline { \tg_t(z) }
+  
i/2\, \dd \wth_t(z) \wedge \dd \overline {\wth_t(z)}  \nonumber \\
&
=
\vartheta \big(\tg_t(z)\big)  \big| 
\tg_t\,'(z)
\big|^2\, 
i\,\dd z\wedge \dd \overline z+  
 \big|\wth_t'(z)\big|^2   \,i/2 \,\dd z\wedge \dd \overline z.  \label{formula-F-omega}
 \end{align}
Taking derivatives on~\eqref{defn-g-t-case1},~\eqref{defn-h-t-case1}, we get  
$$\tg_t\,'=(g_{t-1}^\star)' +\Psi_t' =(g_{t-1}^\star)'+ \sum_{k=1}^{{Q_{\mathfrak{m}}}}  q_{t, k}\cdot  \alpha_t \,\psi_{t,k}^{\alpha_t-1}\psi_{t,k}' , $$
$$   \wth_t'=(h_{t-1}^\star)' +\Phi_t'  =(h_{t-1}^\star)' + 
 \sum_{j=1}^{{P_{\mathfrak{m}}}}  p_{t, j}\cdot \alpha_t \, \varphi_{t,j}  ^{\alpha_{t}-1}\varphi_{t,j}' .$$
Note that  $\varphi_{t, \bullet}$, $\psi_{t, \bullet}$ satisfy the estimates \eqref{condition-varphi},~\eqref{condition-psi}, while  $p_{t, \bullet}, q_{t, \bullet}\leqslant 1$. By much the same argument as \eqref{computation-Phi-t-s}, on $\overline \D_{R_t}$, for $\alpha_t\gg 1$, we have
\begin{equation}\label{bound-Phi-Psi}
|\Phi_t'| \lesssim \alpha_t m^{\alpha_t-1}m', \qquad |\Psi_t'|\lesssim  \alpha_t m^{\alpha_t-1}m'.
\end{equation}
Together with the  boundedness of $\vartheta, (g_{t-1}^\star)', (h_{t-1}^\star)'$  on $\overline \D_{R_t}$, we obtain that
\begin{equation}\label{length-upper-bound-shape}
\Length_{\omega} \big(F_t( \partial \D_{R_t}) \big)\lesssim  R_t \cdot \alpha_t 
 m^{\alpha_t-1}  m'.
 \end{equation}

On the other hand, by Lemma \ref{lem-area-lower-bound-shpae} below, $\Area_\omega (F_{t}(\D_{R_t}))\gtrsim (m^{3/2})^{\alpha_t-1}$. Thus,
 $$ \frac{\Length_\omega (F_{t}(\D_{R_t}))}{\Area_\omega (F_{t}(\D_{R_t}))} \leqslant
 O_t(1)
 \cdot
 \frac{ \alpha_t 
 m^{\alpha_t-1}  m'}{(m^{3/2} )^{\alpha_{t}-1}},   $$
which tends to $0$ as  $\alpha_t\to \infty$. This concludes the proof.
\end{proof}

\begin{lemma}\label{lem-area-lower-bound-shpae}
For $\alpha_t\gg 1$, for every possible $j,k$, one has
\begin{equation*}
    \label{second one, easy analogue}
\int_{\D(\tx_{t,j}, 1)\cap \D_{R_t}} (\pi\circ\wth_t)^*  \omega_E \gtrsim p_{t,j}^2\cdot \alpha_t ^2 (m^{3/2})^{\alpha_t -1},
\end{equation*}
\begin{equation*}
    \label{first one, not easy analogue}
    \int_{\D(\ty_{t,k}, 1)\cap \D_{R_t}}  (\wp\circ\tg_t)^* \omega_\FS \gtrsim  q_{t,k}^2\cdot \alpha_t ^2 (m^{3/2})^{\alpha_t -1}.
    \end{equation*}
\end{lemma}

\begin{proof}
We only prove the second inequality, since the first one is simpler. Recall

\smallskip\noindent
{\bf Bloch's Theorem}
(\cite[p.~295]{Conway}){\bf .}\quad Let $g$ be a holomorphic function on a region containing $\D(z, r)$. Then $g(\D(z, r))$ contains a disc of radius $r|g'(z)|/72$.

\smallskip

For each $k=1,  \dots, Q_{\mathfrak{m}}$,
we can select some point $z_k\in \D_{R_t}\cap \D(\ty_{t, k}, 1)$ with
$\psi_{t, k}'(z_k)\neq 0$ and $|\psi_{t, k}(z_k)|>m^{3/4}$.
Since $|\psi_{t, k'}|<1 $ for other $k'\neq k$ on $\D(\ty_{t, k}, 1)$, we have
\[
|\widehat{g}_t\,'(z_k)|
\geqslant 
|q_{t, k} \cdot \alpha_t \, \psi_{t, k}'(z_k)  \,
\psi_{t, k}^{\alpha_t-1} (z_k)|
-
O_t(1)
\gtrsim
q_{t, k} \cdot \alpha_t
(m^{3/4})^{\alpha_t-1}.
\]
Let $r_k:= (R_t- |z_k|)/2$.
By Bloch's theorem, 
$\widehat{g}_t\big(\D(z_k,r_k )\big)$ contains 
a disc of radius
$\gtrsim  q_{t, k} \cdot \alpha_t
(m^{3/4})^{\alpha_t-1}$.
In particular, this disc strictly contains
at least
$\gtrsim  q_{t, k}^2 \cdot \alpha_t^2
(m^{3/2})^{\alpha_t-1}$ copies of fundamental domains of the lattice $\Gamma$, while each copy contributes some constant positive area  with respect to $\wp^* \omega_{\FS}$.
Thus we have
\begin{equation}
    \label{clever use of Bloch's theorem}
\int_{\D(z_k,r_k )}\,
(\wp\circ\widehat{g}_t)^*\,\omega_{\FS}
\geqslant
\int_{\tg_t(\D(z_k, r_k))}\,
\wp^*\,\omega_{\FS}
\gtrsim
q_{t, k}^2 \cdot \alpha_t^2
(m^{3/2})^{\alpha_t-1}.
\end{equation}
This ends the proof.
\end{proof}

\smallskip

\begin{proof}[Proof of the second inequality of \eqref{length-area conditions all-shape}]   
Recall the definition
$$
\fL_{F_t}(R_t) := \int_0^{R_t}\Length_{\omega} \big( F_t(\partial \D_r)\big)\,\frac{\dd r}{ r}\quad\text{and}\quad  
\fA_{F_t}(R_t) := \int_0 ^{R_t}  \Area_\omega \big(F_{t}(\D_r)\big)
\,\frac{\dd r}{ r}.$$ 
By the same argument as~\eqref{length-upper-bound-shape}, for every $r\leqslant R_t$, we have
$$
\Length_{\omega} \big( F_t(\partial \D_r)\big) \lesssim r\cdot \alpha_t  m^{\alpha_t-1}m'.
$$
Hence
$$\fL_{F_t}(R_t) \lesssim   \int_0^{R_t} r\cdot \alpha_t  m^{\alpha_t-1}m'\,\frac{\dd r}{ r} 
 \lesssim  
 R_t\cdot  \alpha_t  m^{\alpha_t-1}m'.  $$

Now we bound $\fA_{F_t}(R_t)$ from below. Take a $k_0$ such that $q_{t, k_0} \cdot\alpha_{k_0}\geqslant 1$. Let $z_{k_0}, r_{k_0}$ be chosen as in the proof of  Lemma \ref{lem-area-lower-bound-shpae}. Set 
$R_t^-: = |z_{k_0}|+r_{k_0} < R_t$. 
By~\eqref{clever use of Bloch's theorem}, 
we have
$ \Area_\omega \big(F_{t}(\D_r)\big)\gtrsim (m^{3/2} )^{\alpha _{t}-1}$
for all $ r\in [R_t^-, R_t]$. 
Hence 
\begin{equation}\label{lower-bound-N}
\fA_{F_t}(R_t) \geqslant \int_{R_t^-}^{R_t}  \Area_\omega \big(F_{t}(\D_r)\big)
\,\frac{\dd r}{ r} \gtrsim  (m^{3/2} )^{\alpha _{t}-1}.
\end{equation}
Thus
we conclude that 
 $$  \frac{\fL_{F_t}(R_t) }{\fA_{F_t}(R_t)}\leqslant O_t(1)
 \cdot
 {\alpha _t 
 m^{\alpha _t-1}m' \over (m^{3/2} )^{\alpha_{t}-1}},   $$
 which tends to $0$ as  $\alpha_t\to \infty$. This concludes the proof.
\end{proof}

\smallskip

\begin{proof}[Proofs of \eqref{I=0, Ahfors diffuse is trivial, estimate} and \eqref{I=0, Nevanlinna diffuse is trivial, estimate}]
If $z\in \D(\tx_{t,j},1)$ for some $j$, then $|\psi_{t,k}(z)|\leqslant m_0<1$ for every $k$, since
$\overline{\mathbb{D}}(\ty_{t,k}, 1)$ and $\overline{\mathbb{D}}(\tx_{t, j}, 1)$ are disjoint by our construction. This gives
$$|\tg_t (z)- g_{t-1}^\star (z)|\leqslant \sum_{k=1}^{Q_{\mathfrak{m}}} q_{t,k} \big|\psi_{t,k}(z)\big|^{\alpha _t}\ll 1  \quad\text{as}\quad \alpha_t\gg 1.  $$ 
This combining with \eqref{dist-x-h} yields
$\dist_{\omega_{\FS}} \big( x_j,  \wp \circ \tg_t(z) \big)< 2\delta_t$.  Similarly, for $z\in \D(\ty_{t,k} , 1)$, we have $\dist_{\omega_E} \big( y_k,  \pi \circ \wth_t(z) \big)< 2\delta_t$ from \eqref{dist-y-g}. Therefore, 
\begin{equation}\label{D-x-1-subset-x-delta-t}
F_t\big(\D(\tx_{t,j} , 1)\big)\subset \D(x_j,2\delta_t)\times E , \quad     F_t\big(\D(\ty_{t,k} , 1)\big)\subset \C\P^1 \times \D(y_k,2\delta_t) .
\end{equation}
It follows that
$$ \Area_{\omega}\big(F_t(\D_{R_t})\setminus (U_t\cup V_t)\big) \leqslant \Area_{\omega} \big( F_t\big(  \D_{R_t} \setminus  ( \widehat U_t \cup \widehat V_t ) \big)      \big). $$
Observe from \eqref{bound-Phi-Psi} that, for $z\notin \widehat U_t \cup \widehat V_t$, both $|\tg_t'|$ and $|\wth_t'|$ are  $O_t(1)$.  Thus, \eqref{formula-F-omega} implies that  
$\Area_{\omega} \big( F_t\big(  \D_{R_t} \setminus  ( \widehat U_t \cup \widehat V_t ) \big)      \big)  \leqslant O_t(1)$.
Plugging in  Lemma \ref{lem-area-lower-bound-shpae}, we conclude that 
$$ { \Area_{\omega}(F_t(\D_{R_t})\setminus (U_t \cup V_t)) \over \Area_{\omega}(F_t(\D_{R_t}))  }  \lesssim { O_t(1)\over  (m^{3/2} )^{\alpha _{t}-1}  }
.$$
Whence~\eqref{I=0, Ahfors  diffuse is trivial, estimate} follows.

\smallskip

Repeating the same argument, we can likewise obtain that
$$\Area_{\omega}\big(F_t(\D_r)\setminus (U_t \cup V_t)\big) \leqslant r^2\cdot O_t(1).$$
Hence
$$\int_{0} ^{R_t}  
 \Area_{\omega}\big(F_t(\D_r)\setminus (U_t \cup V_t)\big) \, \frac{\dd r}{ r} \leqslant \int_{0} ^{R_t} r\cdot O_t(1)    \,\dd r \leqslant O_t(1). $$
By \eqref{lower-bound-N}, we finish the proof of \eqref{I=0, Nevanlinna diffuse is trivial, estimate}.
\end{proof}

  It remains to prove Proposition \ref{Area distribution}. To achieve this, we state a lemma first. Define
$$\kappa_t:=\int_{ \D( \tx_{t,1},1)\cap \D_{R_t}} \big(\varphi_{t,1}^{\alpha_t} \big)^* (\pi^* \omega_E ) \quad \text{and}\quad \kappa_t^\oN:=   \int_0^{R_t}  \int_{ \D( \tx_{t,1},1)\cap \D_{r}} \big(\varphi_{t,1}^{\alpha_t} \big)^* (\pi^* \omega_E )  \,{\dd r\over r} .  $$

\begin{lemma} \label{lem-area-equal}
There exists a  constant  $c>0$ independent of $t\geqslant 1$, such that for every $j,k$,
$$\Ac_{t,j}=p_{t,j}^2 \cdot \kappa_t +O_t(1) m^{\alpha_t-1}, \quad \Bc_{t,k}=c \,q_{t,k}^2 \cdot \kappa_t +O_t(1) m^{\alpha_t-1},$$
$$\Ac_{t,j}^\oN=p_{t,j}^2 \cdot \kappa_t^\oN +O_t(1) m^{\alpha_t-1}, \quad \Bc_{t,k}^\oN=c \,q_{t,k}^2 \cdot \kappa_t^\oN +O_t(1) m^{\alpha_t-1}.$$
\end{lemma}

\begin{proof}
We prove the equalities for $\Ac_{t,j}$ and $\Bc_{t,k}$. The others are similar.
 Using \eqref{condition-varphi} and \eqref{condition-psi}, and considering the order of $m$,  we  obtain
$$ \Ac_{t,j}=\int_{ \D( \tx_{t,j},1)\cap \D_{R_t}} \big( p_{t,j} \varphi_{t,j}^{\alpha_t} \big)^* (\pi^* \omega_E)+O_t(1) m^{\alpha_t- 1} $$
$$\Bc_{t,k}=\int_{\D(\widehat y_{t,k},1) \cap \D_{R_t}} \big( q_{t,k} \psi_{t,k}^{\alpha_t} \big)^* (\wp^* \omega_{\FS})   +O_t(1) m^{\alpha_t- 1} .$$

Since $\widetilde \wp^*\omega_{\FS}$ and $\omega_E$ are real closed $2$-forms on $E$, by the fact $\dim_{\R}H^2(E, \mathbb{R})=1$ of the top degree de Rham cohomology group of $E$, there exists a constant $c>0$ and a smooth $1$-form $\theta$ such that $\widetilde \wp^*\omega_{\FS}= c\, \omega_E + \dd \theta$. 
Here $c$ is determined by the area relation 
$\int_{E}\widetilde \wp^*\omega_{\FS}=c\cdot\int_{E}\omega_E$. Since $\widetilde{\wp}$ is a $2 : 1$ ramified covering,
$\int_E\widetilde \wp^*\omega_{\FS}=2\int_{\C\P^1}\omega_{\FS}=2$. Thus $c=2/\varrho$. 
Using that $\wp=\widetilde \wp \circ \pi$, we have 
\begin{equation}
    \label{Duval's trick}
\wp^*\omega_\FS= \pi^* (\widetilde \wp^*\omega_\FS)=\pi^*(c\,\omega_E +\dd \theta).
\end{equation}
Hence
\begin{equation}
\label{Duval's trick 2}
    \Bc_{t,k}=c \int_{\D(\widehat y_{t,k},1) \cap \D_{R_t}} \big( q_{t,k} \psi_{t,k}^{\alpha_t} \big)^* (\pi^* \omega_E) +\int_{\D(\widehat y_{t,k},1) \cap \D_{R_t}} \big( q_{t,k} \psi_{t,k}^{\alpha_t} \big)^* (\pi^* \dd \theta)  +O_t(1) m^{\alpha_t- 1}.
\end{equation}
Denote the middle term by $\eta_{t,k}$.
Using again that $\varphi_{t, j}, \psi_{t,k}$
are obtained by composing $\varphi_{t, 1}$ with some rotation, together with the fact that $\pi^*\omega_E$ is a constant $(1,1)$-form, we see that 
$$\Ac_{t,j}= p_{t,j}^2\cdot \kappa_t +O_t(1) m^{\alpha_t- 1} \quad \text{and}\quad \Bc_{t,j} =q_{t,k}^2 \cdot\kappa_t+\eta_{t,k} +O_t(1) m^{\alpha_t- 1}.$$
The proof will be done if we can show that $\eta_{t,k} =O_t(1) m^{\alpha_t-1}$. This just follows by using Stoke's formula and the same proof as~\eqref{length-upper-bound-shape}.
\end{proof}

\begin{proof}[Proof of Proposition~\ref{Area distribution}] Borrow the constant $c$ in Lemma \ref{lem-area-equal}. Take a sufficiently small positive number $c'$ such that all the parameters
\[
p_{t, j}
=
c'
\cdot
\sqrt{A_j},  
\quad
q_{t, k}
=
c'
\cdot
\sqrt{B_k/c}
\]
are in the interval $(0, 1)$.
By Lemma \ref{lem-area-equal}, the Proposition  follows from $\kappa_t, \kappa_t^\oN \gtrsim m^{\alpha_t}$, which is a corollary of~\eqref{clever use of Bloch's theorem} and \eqref{lower-bound-N}.
\end{proof}

\subsection{Proof of Theorem \ref{thm-shape}}

Read $\bA_\fm= (A_1^{[\fm]},\dots, A_{P_\fm}^{[\fm]})$ and $\bB_\fm=(B_1^{[\fm]},\dots, B_{Q_\fm}^{[\fm]})$.

\begin{proposition} \label{prop-shape}
For any $\mathfrak{m}\geqslant 1$, 
there exists an increasing sequence $\{t_s\}_{s\geqslant 1}$ in $\Z_+$ such that the sequence of holomorphic discs $\{\mathbf F \,\lvert_{\overline \D_{R_{t_s}}}\}_{s \geqslant 1}$ produces  Nevanlinna/Ahlfors current 
\begin{equation*}\label{alg-E-R}
\oT= \sum_{j=1}^{P_{\mathfrak{m}}} {A^{[\mathfrak{m}]}_j \over \varrho} \cdot [\mathsf{E}_j]
+\sum_{k=1}^{Q_{\mathfrak{m}}} B^{[\mathfrak{m}]}_k \cdot [\mathsf{R}_k],
\end{equation*}
where $\mathsf{E}_j:=\pi_1^{-1}(x_j)$ are elliptic curves, and where
$\mathsf{R}_k:=\pi_2^{-1}(y_k)$ are  rational curves.
\end{proposition}

\begin{proof}
Since $\widehat{\mathbf F}=\lim_{j\rightarrow \infty} \widehat F_j$,
for every $t\geqslant 1$, on $\D_{R_t+1}$ we have
$$
\|\widehat{\mathbf F}-F_t\|
\leqslant
\sum_{j\geqslant t}\,
\|F_{j+1}-F_j\|
\leqslant
\sum_{j\geqslant t}\,
2^{-j-1}\cdot \epsilon_{j}
<
\sum_{j\geqslant t}\,
2^{-j-1}\cdot \epsilon_{t}
<
\epsilon_{t}.
$$
Hence by the stability condition ($\diamondsuit$), 
the entire curve $\mathbf F=\Pi\circ \widehat{\mathbf F}$ satisfies all the estimates in Propositions \ref{key-prop-shape} and \ref{Area distribution}  after replacing $2^{-t}$ by $2^{-t+1}$.

\smallskip
Write  elements of $\sigma^{-1}(\fm)\subset \mathbb{Z}_+$ in the increasing order as $\{t_{ s}\}_{s\geqslant 1}$. 
Due to the length-area estimate~\eqref{length-area conditions all-shape}, after passing to certain subsequence of  $\{\mathbf F \,\lvert_{ \overline \D_{R_{t_{ s}}}}\}_{s\geqslant 1}$,
we can obtain some Nevanlinna/Ahlfors current $\mathcal{T}$. These two currents happen to coincide. We only prove the Ahlfors part, while the argument works for Nevanlinna part as well.
By definition, the mass of $\oT$ on $\mathsf{E}_{j_0}=\{x_{j_0} \} \times E$ is equal to
$$\int_{\{x_{j_0}\} \times E} \oT \wedge \omega= \lim_{s\to\infty } \int_{ \D (x_{j_0}, 2\delta_{t_{ s}}) \times E } \oT \wedge \omega   =\lim_{s\to\infty} {\Area_{\omega} (\mathbf F[ \mathbf F^{-1}( \D (x_{j_0}, 2\delta_{t_{s}}) \times E  )     ] ) \over \Area_{\omega}(\mathbf F(\D_{R_{t_{ s}}}))       } $$
Recall~\eqref{D-x-1-subset-x-delta-t} that $\mathbf F\big(\D(\tx_{t_{ s},j_0} , 1 )\big)\subset \D(x_{j_0},2\delta_{t_{ s}})\times E $. It follows that 
$$ 
\int_{\{x_{j_0}\} \times E} \oT \wedge \omega
\geqslant 
\lim_{s\to\infty}  {\Ac_{t_s, j_0} \over \Area_{\omega}(\mathbf F(\D_{R_{t_{ s}}})) } 
= 
\frac{A^{[\mathfrak{m}]}_{j_0}}{\sum_{j=1}^{P_{\fm}} A^{[\mathfrak{m}]}_{j}
+\sum_{k=1}^{Q_{\fm}} B^{[\mathfrak{m}]}_{k}
}
$$
by Proposition \ref{Area distribution} and~\eqref{I=0, Ahfors diffuse is trivial, estimate}.
Similarly, the mass of $\oT$ on $\mathsf{R}_{k_0}=\C\P^1 \times \{y_{k_0}\}$ is
$$
\int_{\C\P^1 \times \{y_{k_0}\}} \oT \wedge \omega 
\geqslant 
\frac{B^{[\mathfrak{m}]}_{k_0}}{\sum_{j=1}^{P_{\fm}} A^{[\mathfrak{m}]}_{j}
+\sum_{k=1}^{Q_{\fm}} B^{[\mathfrak{m}]}_{k}
}.$$
Since the total mass  of $\oT$ is  $1$, while $\oT$ cannot charge positive mass on points $\{x_{j_0}\}\times\{y_{k_0}\}$,   
the above inequalities must be exactly equalities. Hence we conclude the proof.
\end{proof}

\begin{lemma}
\label{A-C trick again}
    Let $f: \mathbb{C}\rightarrow X$ be an entire curve which produces Nevanlinna/Ahlfors currents $\{\oT_s\}_{s=1}^\infty$ by $\{f\,\lvert _{\overline{\D}_r}\}_{r>0}$. If
     $\{\oT_s\}_{s=1}^\infty$ converges weakly to a current $\oT$,
   then $\oT$ is also a Nevanlinna/Ahlfors current generated  by $\{f\,\lvert _{\overline{\D}_r}\}_{r>0}$.
\end{lemma}

\begin{proof}
We only prove the case of Ahlfors currents, since the other case is similar. 
For each $s\geqslant 1$,
let $\oT_s$ be generated by the radii $\{r_{s,\ell}\}_{\ell= 1}^\infty$. By throwing some smaller radii,
we can assume that for all $\ell\geqslant 1$, the  length-area ratios are small
\begin{equation}
    \label{the last l/A condition}
\dfrac{\Length_{\omega}(f(\partial \D_{r_{s,\ell}}))}{\Area _{\omega}(f( \D_{r_{s,\ell}}))}
<
2^{-s}.
\end{equation}
  In the Fr\'echet space
$\mathcal{A}^{1, 1}(X)$ of
smooth $(1, 1)$-forms
on $X$, we can 
take a countable dense subset $\{\theta_j\}_{j= 1}^\infty$.
For each $j\geqslant 1$ and  $\epsilon>0$,
we can choose a large integer $M_{j, \epsilon}$ such that
$|\lp\oT_s,\theta_j\rp-\lp\oT,\theta_j\rp|<\epsilon$ for all $s\geqslant M_{j, \epsilon}$. Moreover, 
by definition of $\oT_s$, 
we can find $N_{j, s, \epsilon}\gg 1$ such that
for all $\ell\geqslant N_{j, s, \epsilon}$, there hold
\[
\Big|
\frac{1}{ \Area_{\omega} (f(\D_{r_{s,\ell}}))} \, \lp f_*  [\D_{r_{s,\ell}}] 
,\theta_j\rp
-
\lp\oT_s,\theta_j\rp
\Big|
< \epsilon.
\]
Hence for all $j\geqslant 1$, $\epsilon>0$,  $s\geqslant M_{j, \epsilon}$, $\ell\geqslant N_{j, s, \epsilon}$, we have
\begin{equation}
    \label{key A-A condition}
\Big|
\frac{1}{ \Area_{\omega} (f(\D_{r_{s,\ell}}))} \, \lp f_*  [\D_{r_{s,\ell}}] 
,\theta_j\rp
-
\lp\oT,\theta_j \rp
\Big|
<2 \epsilon.
\end{equation}
By the diagonal argument of the Arzel\`a-Ascoli Theorem, from $\{r _{s,\ell}\}_{s, \ell\geqslant 1}$
we can extract a sequence of increasing radii
$\{r_k\}_{k= 1}^\infty$ tending to infinity, 
such that
the holomorphic discs $\{f\,\lvert_{\overline\D_{r_k}}\}_{k\geqslant 1}$ have length-area ratios tending to zero as guaranteed by~\eqref{the last l/A condition}, and that they generate $\oT$
as an Ahlfors current 
by~\eqref{key A-A condition}.
\end{proof}

We have the following two observations.

\medskip\noindent{\bf Observation 1.} 
    Since $\mathbb{Q}_+\subset \mathbb{R}_{\geqslant 0}$ is dense, by Proposition~\ref{prop-shape} and Lemma~\ref{A-C trick again}, for all real numbers $a_j, b_k\geqslant 0$ with $\sum_{j=1}^{\infty} a_j+\sum_{k=1}^{\infty} b_k=1$, $\{\mathbf F \,\lvert_{\overline \D_{r}}\}_{r>0}$ can generate 
    Nevanlinna/Ahlfors currents 
    \[
\oT= \sum_{j=1}^{\infty} {a_j\over \varrho} \cdot [\pi_1^{-1}(x_j)]
+\sum_{k=1}^{\infty} b_k \cdot [\pi_2^{-1}(y_k)].
\]

\noindent{\bf Observation 2.} 
    Since both $\{x_j\}_{j\geqslant 1}\subset \C\P^1$ and $\{y_k\}_{k\geqslant 1}\subset E$ are dense, by \textbf{Observation 1}  above and Lemma~\ref{A-C trick again}, 
    for arbitrary points
 $\{x'_j\}_{j=1}^{\infty}\subset \C\P^1$ and $\{y'_k\}_{k=1}^{\infty}\subset E$, for all real numbers $a_j, b_k\geqslant 0$ with $\sum_{j=1}^{\infty} a_j+\sum_{k=1}^{\infty} b_k=1$, 
    $\{\mathbf F \,\lvert_{\overline \D_{r}}\}_{r>0}$ can generate 
    Nevanlinna/Ahlfors currents 
      \begin{equation*}
      \label{arbitrary shape N/A currents}
\oT= \sum_{j=1}^{\infty} {a_j \over \varrho} \cdot [\pi_1^{-1}(x'_j)]
+\sum_{k=1}^{\infty} b_k \cdot [\pi_2^{-1}( y'_k)].
\end{equation*}

 \smallskip

\begin{proof}[Proof of Theorem~~\ref{thm-shape}]
\textbf{Observation 2} above shows that $\mathbf F$ can produce all the shapes of Nevanlinna/Ahlfors currents with trivial diffuse part.
It remains to show examples of the shape $(|J_{\mathsf{E}}|\in \mathbb{Z}_{\geqslant 0} \cup \{\infty\}, |J_{\mathsf{R}}| \in \mathbb{Z}_{\geqslant 0} \cup \{\infty\}, \oT_{\dif} \text{ is nontrivial})$.

For every $s\geqslant 1$, select some ``equidistributed'' points $\{w_{s, \ell}\}_{\ell=1}^s\subset E$ such that the sequence of averaged Dirac measures 
$\mu_s:= \frac{1}{s}\sum_{\ell=1}^s \delta_{w_{s,\ell}}$ converges weakly to $\omega_E/\varrho$  as $s$ tends to infinity.
Choose arbitrary distinct points $\{x'_j\}_{j=1}^{|J_{\mathsf{E}}|}\subset \C\P^1$, $\{y'_k\}_{k=1}^{|J_{\mathsf{R}}|}\subset E$.
Take strictly positive real coefficients 
$\{a_j\}_{j=1}^{|J_{\mathsf{E}}|}, \{b_k\}_{k=1}^{|J_{\mathsf{R}}|}$
with total sum $\sum_{j=1}^{|J_{\mathsf{E}}|} a_j+\sum_{k=1}^{|J_{\mathsf{R}}|} b_k<1$. By \textbf{Observation 2}, 
    $\{\mathbf F \,\lvert_{\overline \D_{r}}\}_{r>0}$ can generate 
    Nevanlinna/Ahlfors currents 
\begin{equation}
    \label{T_s last game}
\oT_s
:=
\Big(\sum_{j=1}^{\min \{s, |J_{\mathsf{E}}|\}} {a_{j}\over \varrho} \cdot [\mathsf \pi_1^{-1}(x_j')] + \sum_{k=1}^{\min \{s, |J_{\mathsf{R}}|\}} b_{k}\cdot [\mathsf \pi_2^{-1}(y_k')]
\Big)
+
\Big(
\frac{c_s}{s}\cdot
\sum_{\ell=1}^{s}  
[\mathsf \pi_2^{-1}(w_{s, \ell})]
\Big),
\end{equation}
where $c_s:=1-\sum_{j=1}^{\min\{s, |J_{\mathsf{E}}|\}}a_{j}-\sum_{k=1}^{\min\{s, |J_{\mathsf{R}}|\}} b_{k} >0$.
As $s$ tends to infinity,
the first bracket $(\cdots)$ of~\eqref{T_s last game} converges weakly to a singular current
\[
\oT_{\alg}
:=
\sum_{j=1}^{|J_{\mathsf{E}}|} { a_{j}\over \varrho }\cdot [\mathsf \pi_1^{-1}(x_j')] + \sum_{k=1}^{|J_{\mathsf{R}}|} b_{k}\cdot [\mathsf \pi_2^{-1}(y_k')],
\]
while the second 
bracket $(\cdots)$ of~\eqref{T_s last game} converges weakly to a nontrivial diffuse current $\oT_{\dif}$. Indeed, it is easy to check that, for any smooth $(1, 1)$-form $\eta$ on $X$, 
$$
\lp \oT_{\dif},\eta \rp
=
\Big(
1-\sum_{j=1}^{ |J_{\mathsf{E}}|}a_{j}-\sum_{k=1}^{ |J_{\mathsf{R}}|} b_{k}
\Big) \cdot{1\over \varrho}\int_{X} \eta\wedge \pi_2^*\omega_E.$$ 
Thus by  Lemma~\ref{A-C trick again}, we obtain 
a Nevanlinna/Ahlfors current $\oT:=\oT_{\alg}+\oT_{\dif}$  with the desired shape.
\end{proof}

\end{document}